\documentclass[a4paper, notitlepage]{article}

\usepackage[latin1]{inputenc}
\usepackage[british,english]{babel}
\usepackage{calc,amssymb,mathtools,amsthm,stmaryrd,units,enumitem,tabularx,booktabs,multicol}
\usepackage[strict]{chngpage}
\mathtoolsset{mathic}

\usepackage{hyperref}
\hypersetup{
breaklinks=true,
colorlinks=true,
linkcolor=black,
anchorcolor=black,
citecolor=black,
filecolor=black,
menucolor=black,
pagecolor=black,
urlcolor=black
}
\usepackage[alphabetic,non-sorted-cites]{amsrefs}

\usepackage[all]{xy}
\newcommand{\ie}{i.\,e.\ }
\newcommand{\eg}{e.\,g.\ }

\chardef\bslash=`\\

\theoremstyle{plain}
   \newtheorem{thm}{Theorem}[section]
   \newtheorem{prop}[thm]{Proposition}
   \newtheorem{cor}[thm]{Corollary}
   \newtheorem*{cor*}{Corollary}
   \newtheorem{lem}[thm]{Lemma}
   
   \newtheorem{myAxioms}[thm]{Standing assumptions}
\theoremstyle{definition}
   \newtheorem{defn}[thm]{Definition}
   \newtheorem*{defn*}{Definition}
   \newtheorem*{lem*}{Lemma}
\theoremstyle{remark}
   
   \newtheorem{rem}{Remark}[section]
   \newtheorem*{rem*}{Remark}
   
\newcommand{\Z}{\mathbb{Z}}
\newcommand{\R}{\mathbb{R}}
\newcommand{\C}{\mathbb{C}}
\newcommand{\Q}{\mathbb{Q}}
\newcommand{\HH}{\mathbb{H}}

\DeclareMathOperator{\rank}{rank}
\DeclareMathOperator{\GL}{GL}

\DeclareMathOperator{\coker}{coker}
\DeclareMathOperator{\colim}{colim}

\newcommand{\smsh}{\wedge}
\newcommand{\N}{\mathcal N}

\DeclareMathOperator{\Thom}{Thom}
\DeclareMathOperator{\point}{point}
\DeclareMathOperator{\BOspace}{BO}
        \newcommand{\BO}{\ensuremath{\BOspace}}
\DeclareMathOperator{\BUspace}{BU}
        \newcommand{\BU}{\ensuremath{\BUspace}}
\DeclareMathOperator{\Uspace}{U}
        \newcommand{\U}{\ensuremath{\Uspace}}

\DeclareMathOperator{\Pic}{Pic}

\newcommand{\A}{\mathbb{A}}
\newcommand{\G}{\mathbb{G}}
\renewcommand{\P}{\mathbb{P}}
        \newcommand{\CP}{{\mathbb{C}\mathbb{P}}}
\newcommand{\OO}{\mathcal{O}}
\newcommand{\HOM}{\mathcal{H}{om}}
\renewcommand{\tilde}{\widetilde}

\newcommand{\rH}{\widetilde{H}}
\newcommand{\rE}{\widetilde{E}}
\DeclareMathOperator{\Kgroup}{K}
        \newcommand{\K}{\ensuremath{\Kgroup}}
\DeclareMathOperator{\KOgroup}{KO}

        \newcommand{\KO}{\ensuremath{\KOgroup}}
        \newcommand{\rK}{\ensuremath{\widetilde{\K}{}} }
        \newcommand{\Kh}{\ensuremath{\KOgroup}}
        \newcommand{\rKh}{\ensuremath{\widetilde{\Kh}{}}}
        \newcommand{\rKO}{\ensuremath{\widetilde{\KO}{}} }

\newcommand{\quotient}[2]{\raisebox{2pt}{\( #1 \)}\!\Big/\raisebox{-2pt}{\( #2 \)}}

    \DeclareMathOperator{\Wgroup}{W}
        \newcommand{\W}{\Wgroup}
        
        \DeclareMathOperator{\w}{w}
        \DeclareMathOperator{\GWgroup}{GW}
        \newcommand{\GW}{\GWgroup}
        
        \DeclareMathOperator{\gw}{gw}
\newcommand{\lb}[1]{{\mathcal{#1}}}
\newcommand{\vb}[1]{{\mathcal{#1}}}
\newcommand{\dual}{\vee}

\newcommand{\SgmInf}{\Sigma^{\infty}}

\newcommand{\SH}{\mathcal{SH}}
\newcommand{\pH}{\mathcal{H}_{\bullet}}
\newcommand{\uH}{\mathcal{H}}
\newcommand{\SPaK}{{\mathbb{K}}}
\newcommand{\SPtK}{\mathbb{K}^{\text{top}}}
\newcommand{\SPaKO}{\mathbb{K}\mathbf{O}}
\newcommand{\SPtKO}{\mathbb{K}\mathbf{O}^{\text{top}}}
\newcommand{\SPGW}{\mathbb{G}\mathrm{W}}
\newcommand{\mm}[1]{\left(\begin{smallmatrix}#1\end{smallmatrix}\right)}

\DeclareMathOperator{\Gr}{Gr}
 \newcommand{\Grnd}{\Gr^{\text{nd}}}
\DeclareMathOperator{\SO}{SO}
\DeclareMathOperator{\SU}{SU}
\DeclareMathOperator{\SL}{SL}
\DeclareMathOperator{\Sp}{Sp}

\DeclareMathOperator{\Sq}{Sq}
\DeclareMathOperator{\Spin}{Spin}
\newcommand*{\longhookrightarrow}{\ensuremath{\lhook\joinrel\relbar\joinrel\relbar\joinrel\rightarrow}}
\newcommand*{\longhookleftarrow}{\ensuremath{\leftarrow\joinrel\relbar\joinrel\relbar\joinrel\rhook}}

\newcolumntype{M}{>{$}c<{$}}
\newcommand{\minrowheight}[1]{\rule[\heightof{#1}-\totalheightof{#1}]{0pt}{\totalheightof{#1}+3pt}}
\newcommand{\dummyfrac}{\ensuremath{\frac{\big(N\big)}{N}}}

\entrymodifiers={+!!<0pt,\fontdimen22\textfont2>}
\SelectTips{cm}{10}
\begin{document}
\title{Witt groups of complex cellular varieties}
\author{Marcus Zibrowius%
 \thanks{Department of Pure Mathematics and Mathematical Statistics,
  University of Cambridge,
  United Kingdom}
}
\date{27 June 2011\footnote{This is the final version, identical in content to the paper published in \href{http://www.math.uiuc.edu/documenta/vol-16/vol-16-eng.html}{\textit{Documenta Math.}~16 (2011)} on pages 465--511.}}
\maketitle
\begin{abstract}
\noindent We show that the Grothendieck-Witt and Witt groups of smooth complex cellular varieties can be identified with their topological KO-groups. As an application, we deduce the values of the Witt groups of all irreducible hermitian symmetric spaces, including smooth complex quadrics, spinor varieties and symplectic Grassmannians.
\end{abstract}
\tableofcontents
\thispagestyle{empty}
\setlength{\parindent}{0pt}
\addtolength{\parskip}{3pt}
\addtolength{\topsep}{3pt}

\section*{Introduction}\label{sec:intro}
\addcontentsline{toc}{section}{Introduction}
\thispagestyle{empty}
The purpose of this paper is to demonstrate that the Grothendieck-Witt and Witt groups of complex projective homogeneous varieties can be computed in a purely topological way. That is, we show in Theorem~\ref{thm:mainthm} how to identify them with the topological KO-groups of these varieties, and we illustrate this with a series of known and new examples.

Our theorem holds more generally for any smooth complex cellular variety. By this we mean a smooth complex variety \( X \) with a filtration by closed subvarieties \( {\emptyset=Z_0\subset Z_1\subset Z_2 \dots\subset Z_N=X} \) such that the complement of \( Z_k \) in \( Z_{k+1} \) is an open ``cell'' isomorphic to \( \A^{n_k} \) for some \( n_k \). Let us put our result into perspective. It is well-known that for such cellular \( X \) we have an isomorphism
\begin{align*}
 \K_0(X)\overset{\cong}\longrightarrow\K^0(X(\C))
\end{align*}
between the algebraic K-group of \( X \) and the complex K-group of the underlying topological space \( X(\C) \). In fact, both sides are easy to compute: they decompose as direct sums of the K-groups of the cells, each of which is isomorphic to \( \Z \). Such decompositions are characteristic of oriented cohomology theories.
Witt groups, however, are strictly non-oriented, and this makes computations much harder. It is true that the Witt groups of complex varieties decompose into copies of \( \Z/2 \), the Witt group of \( \C \), but even in the cellular case there is no general understanding of how many copies to expect.

Nonetheless, we can prove our theorem by an induction over the number of cells of \( X \). The main issue is to define the map from Witt groups to the relevant KO-groups in such a way that it respects various exact sequences. The basic idea is clear: the Witt group \( \W^0(X) \) classifies vector bundles equipped with non-degenerate symmetric forms, and in topology symmetric complex vector bundles are in one-to-one correspondence with real vector bundles, classified by \( \KO^0(X) \). More precisely, we have two natural maps:
\begin{align*}
 \GW^0(X) & \rightarrow \KO^{0}(X(\C))\\
 \W^0(X)  & \rightarrow \tfrac{\KO^{0}(X(\C))}{\K^0(X(\C))}
\end{align*}
Here, \( \GW^0(X) \) is the Grothendieck-Witt group of \( X \), and in the second line \( \K^0(X) \) is mapped to \( \KO^0(X) \) by sending a complex vector bundle to the underlying real bundle. It is possible to extend these maps to shifted groups and groups with support in a concrete and ``elementary'' way, as was done in \cite{Me:RSK-Essay}. The method advocated here is to rely instead on a result in \( \A^1 \)-homotopy theory: the representability of hermitian K-theory by a spectrum whose complex realization is the usual topological KO-spectrum. Currently, our only reference is a draft paper of Morel \cite{Morel}, but the result is well-known to the experts and a full published account will undoubtedly become available in due course. In the unstable homotopy category at least, the statement is immediate from Schlichting and Tripathi's recent description of a geometric representing space for hermitian K-theory (see Section~\ref{sec:Representing_algK}).

The structure of the paper is as follows: In the first section we assemble the basic definitions, reviewing some representability results along the way before finally stating in \ref{myAxioms} the results in \( \A^1 \)-homotopy theory that we  ultimately take as our starting point. Our main result, Theorem~\ref{thm:mainthm}, is stated and proved in the second section. Section~\ref{sec:AHSS} reviews mostly well-known facts about the Atiyah-Hirzebruch spectral sequence, on which the computations of examples in the final section rely.

\section{Preliminaries}\label{sec:set-up}
\subsection{Witt groups and hermitian K-theory}\label{subsec:W}
From a modern point of view, the theory of Witt groups represents a K-theoretic approach to the study of quadratic forms. We briefly run through some of the basic definitions.

Recall that the algebraic K-group \( \K_0(X) \) of a scheme \( X \) can be defined as the free abelian group on isomorphism classes of vector bundles over \( X \) modulo the following relation: for any short exact sequence of vector bundles
\begin{equation*}
 0\rightarrow\vb{E}\rightarrow\vb{F}\rightarrow\vb{G}\rightarrow 0
\end{equation*}
over \( X \) we have \( [\vb{F}]=[\vb{E}]+[\vb{G}] \) in \( \K_0(X) \). In particular, as far as \( \K_0(X) \) is concerned, we may pretend that all exact sequences of vector bundles over \( X \) split.

Now let \( (\vb{E},\epsilon) \) be a symmetric vector bundle, by which we mean a vector bundle \( \vb{E} \) equipped with a non-degenerate symmetric bilinear form \( \epsilon \). We may view \( \epsilon \) as an isomorphism from \( \vb{E} \) to its dual bundle \( \vb{E}^\dual \), in which case its symmetry may be expressed by saying that \( \epsilon \) and \( \epsilon^\dual \) agree under the canonical identification of the double-dual \( (\vb{E}^\dual)^\dual \) with \( \vb{E} \).
Two symmetric vector bundles \( (\vb{E},\epsilon) \) and \( (\vb{F},\phi) \) are isometric if there is an isomorphism of vector bundles \( i\colon{\vb{E}\rightarrow\vb{F}} \) compatible with the symmetries, \ie such that \( i^\dual\phi i= \epsilon \). The orthogonal sum of two symmetric bundles has the obvious definition \( {(\vb{E},\epsilon) \perp (\vb{F},\phi)}:={(\vb{E}\oplus\vb{F},\epsilon\oplus\phi)} \).

Any vector bundle \( \vb{E} \) gives rise to a symmetric bundle \( H(\vb{E}):=(\vb{E}\oplus\vb{E}^\dual,\mm{0&1\\1&0}) \), the hyperbolic bundle associated with \( \vb{E} \). These hyperbolic bundles are the simplest members of a wider class of so-called metabolic bundles: symmetric bundles \( (\vb{M},\mu) \) which contain a subbundle \( j\colon{\vb{L}\rightarrow\vb{M}} \) of half their own rank on which \( \mu \) vanishes. In other words, \( (\vb{M},\mu) \) is metabolic if it fits into a short exact sequence of the form
\begin{equation*}
 0\rightarrow \vb{L} \overset{j}\longrightarrow \vb{M} \overset{j^\dual\mu}\longrightarrow \vb{L^\dual} \rightarrow 0
\end{equation*}
The subbundle \( \vb{L} \) is then called a Lagrangian of \( \vb{M} \). If the sequence splits, \( (\vb{M},\mu) \) is isometric to \( H(\lb{L}) \), at least in any characteristic other than two. This motivates the definition of the Grothendieck-Witt group.

\begin{defn}[\cites{Walter:TGW,Schlichting:Kh-exact}]
 The Grothendieck-Witt group \( \GW^0(X) \) of a scheme \( X \) is the free abelian group on isometry classes of symmetric vector bundles over \( X \) modulo the following two relations:
\begin{itemize}%
[topsep=0pt]
  \item \( [(\vb{E},\epsilon)\perp(\vb{G},\gamma)] = [(\vb{E},\epsilon)] + [(\vb{G},\gamma)] \)
  \item \( [(M,\mu)]=[H(\vb{L})] \) for any metabolic bundle \( (M,\mu) \) with Lagrangian \( \vb{L} \)
 \end{itemize}
 The Witt group \( \W^0(X) \) is defined similarly, except that the second relation reads \( [(M,\mu)]=0 \).
 Equivalently, we may define \( \W^0(X) \) by the exact sequence
 \begin{equation*}\label{seq:Karoubi_def}
  \K_0(X)\overset{H}\longrightarrow\GW^0(X)\longrightarrow\W^0(X)\rightarrow 0
 \end{equation*}
\end{defn}

\paragraph{Shifted Witt groups.}
The groups above can be defined more generally in the context of exact or triangulated categories with dualities. The previous definitions are then recovered by considering the category of vector bundles over \( X \) or its bounded derived category. However, the abstract point of view allows for greater flexibility. In particular, a number of useful variants of Witt groups can be introduced by passing to related categories or dualities. For example, if we take a line bundle \( \lb{L} \) over \( X \) and replace the usual duality \( \vb{E}^\dual:=\HOM(\vb{E},\vb{O}_X) \) on vector bundles by \( \HOM(-,\vb{L}) \) we obtain ``twisted'' Witt groups \( \W^0(X;\lb{L}) \). On the bounded derived category, we can consider dualities that involve shifting complexes, leading to the definition of ``shifted'' Witt groups \( \W^i(X) \). This approach, pioneered by Paul Balmer in \cites{Balmer:TWGI,Balmer:TWGII}, elevates the theory of Witt groups into the realm of cohomology theories. We illustrate the meaning and significance of these remarks with a few of the key properties of the theory, concentrating on the case when \( X \) is a smooth scheme over a field of characteristic not equal to two. The interested but unacquainted reader may prefer to consult \cite{Balmer:Ojanguren60} or \cite{Balmer:Handbook}.
\begin{itemize}
 \item For any line bundle \( \lb{L} \) over \( X \) and any integer \( i \), we have a Witt group
     \[\W^i(X;\lb{L})\]
       This is the \( i^{\text{th}} \) Witt group of \( X \) ``with coefficients in \( \lb{L} \)'', or ``twisted by \( \lb{L} \)''.
       When \( \lb{L} \) is trivial it is frequently dropped from the notation.
 \item The Witt groups are four-periodic in \( i \) and ``two-periodic in \( \lb{L} \)'' in the sense that, for any \( i \) and any line bundles \( \lb{L} \) and \( \lb{M} \) over \( X \), we have canonical isomorphisms
       \begin{align*}
        \W^i(X;\lb{L})&\cong\W^{i+4}(X;\lb{L})\\
        \W^i(X;\lb{L})&\cong\W^i(X;\lb{L}\otimes\lb{M}^{\otimes 2})
       \end{align*}
 \item More generally, for any closed subset \( Z \) of \( X \) we have Witt groups ``with support on Z'', written \( \W_Z^i(X;\lb{L}) \). For \( Z=X \) these agree with \( \W^i(X;\lb{L}) \).
 \item We have long exact ``localization sequences'' relating the Witt groups of \( X \) and \( X-Z \), which can
       be arranged as 12-term exact loops by periodicity.
\end{itemize}
Balmer's approach already works on the level of Grothendieck-Witt groups, as shown in \cite{Walter:TGW}. In this context, the localization sequences take the form
\begin{equation}\label{seq:GW-localization}
\begin{aligned}
 & \GW^i_Z(X) \rightarrow \GW^i(X) \rightarrow \GW^i(X-Z) \rightarrow \\
 & \quad \W^{i+1}_Z(X) \rightarrow \W^{i+1}(X) \rightarrow \W^{i+1}(X-Z) \rightarrow \W^{i+2}_Z(X) \rightarrow \cdots
\end{aligned}
\end{equation}
continuing to the right with shifted Witt groups of \( X \), and similarly for arbitrary twists \( \lb{L} \) \cite{Walter:TGW}*{Theorem~2.4}. However, if one wishes to continue the sequences to the left, one has to revert to the methods of higher algebraic K-theory.

\paragraph{Hermitian K-theory.}\label{sec:hermitian_K}
Recall that the higher algebraic K-groups of a scheme \( X \) can be defined as the homotopy groups of a topological space \( \K(X) \) associated with \( X \). If one replaces \( \K(X) \) by an appropriate spectrum one can similarly define groups \( \K_n(X) \) in all degrees \( n\in\Z \).
On a smooth scheme \( X \), however, the groups in negative degrees vanish.

An analogous construction of hermitian K-theory is developed in \cite{Schlichting:MV}. Given a scheme \( X \) and a line bundle \( \lb{L} \) over \( X \), Schlichting constructs a family of spectra \( \SPGW^i(X;\lb{L}) \) from which hermitian K-groups can be defined as
\begin{equation*}
 \GW^i_n(X;\lb{L}) :=\pi_n(\SPGW^i(X;\lb{L}))
\end{equation*}
In degree \( n=0 \), one recovers Balmer and Walter's Grothendieck-Witt groups, and the Witt groups appear as hermitian K-groups in negative degrees. To be precise, for any smooth scheme \( X \) over a field of characteristic not equal to two one has the following natural identifications:
\begin{align}\label{eq:Schlichting_vs_BW}
                                 \GW^i_0(X;\lb{L})    &\cong \GW^i(X;\lb{L}) \\
  \phantom{\text{ for \( n<0 \)}}    \GW^i_n(X;\lb{L})    &\cong \W^{i-n}(X;\lb{L}) \text{ for \( n<0 \)}
\end{align}
For affine varieties, the identifications of Witt groups may be found in \cite{Hornbostel}: see Proposition~A.4 and Corollary~A.5. For a general smooth scheme \( X\), we can pass to a vector bundle torsor \( T\) over \(X\) such that \( T\) is affine \citelist{\cite{Jouanolou}*{Lemma~1.5}\cite{Hornbostel}*{Lemma~2.1}}.\footnote{%
This step is known as Jouanolou's trick.}
Both Balmer's Witt groups and Schlichting's hermitian K-groups are homotopy invariant in the sense that the groups of \( T \) may naturally be identified with those of \( X \). This is proved for Witt groups in \cite{Gille:homotopy-invariance}*{Corollary~4.2} and may be deduced for hermitian K-theory from the Mayer-Vietoris sequences established in \cite{Schlichting:MV}*{Theorem~1}. The identifications also hold more generally for hermitian K-groups with support \( \GW^i_{n,Z}(X) \) \cite{Schlichting:Karoubi}. They will be used implicitly throughout.

For completeness, we mention how the \( 4 \)-periodic notation used here translates into the traditional notation in terms of \( \KO \)- and \( \mathrm{U} \)-theory, as used for example in \cite{Hornbostel}. Namely, we have
\begin{align*}
 \GW^i_n(X) = \begin{cases}
               \,\,\,\,\KO_n(X)        & \text{ for \( i\equiv \;\;\; 0  \mod 4 \)}\\
               \,\,\,\,\mathrm{U}_n(X) & \text{ for \( i\equiv -1 \)}\\
               _{-}\!\KO_n(X)          & \text{ for \( i\equiv -2 \)}\\
               _{-}\mathrm{U}_n(X)     & \text{ for \( i\equiv -3 \)}\\
              \end{cases}
\end{align*}
(This notation will not be used elsewhere in this paper.)

\subsection{KO-theory}\label{subsec:KO}
We now turn to the corresponding theories in topology. To ensure that the definitions given here are consistent with the literature, we restrict our attention to finite-dimensional CW complexes.\footnote{The key property we need is that any vector bundle over a finite-dimensional CW complex has a stable inverse. See the proof of Theorem~\ref{thm:Unstable_Representability_of_K}.}  Since we are ultimately only interested in topological spaces that arise as complex varieties, this is not a problem. The definitions of \( \K_0 \) and \( \GW^0 \) given above applied to complex vector bundles over a finite-dimensional CW complex \( X \) yield its complex and real topological K-groups \( \K^0(X) \) and \( \KO^0(X) \). Since short exact sequences of vector bundles over \( X \) always split, the definitions may even be simplified:
\begin{defn}\label{def:K_and_KO-concrete}
For a finite-dimensional CW complex \( X \), the complex K-group \( \K^0(X) \) is the free abelian group on isomorphism classes of complex vector bundles over \( X \) modulo the relation \( [\vb{E}\oplus\vb{G}]=[\vb{E}]+[\vb{G}] \). Likewise, the KO-group \( \KO^0(X) \) is the free abelian group on isometry classes of symmetric complex vector bundles over \( X \) modulo the relation
 \( [(\vb{E},\epsilon)\perp(\vb{G},\gamma)]=[(\vb{E},\epsilon)]+[(\vb{G},\gamma)] \).
\end{defn}

There is a more common description of \( \KO^0(X) \) as the K-group of real vector bundles. The equivalence with the definition given here can be traced back to the fact that the orthogonal group \( \mathrm{O}(n) \) is a maximal compact subgroup of both \( \GL_n(\R) \) and \( \mathrm{O}_n(\C) \), but also seen very concretely along the following lines. We say that a complex bilinear form \( \epsilon \) on a real vector bundle \( \vb{F} \) is real if \( \epsilon\colon{\vb{F}\otimes\vb{F}\rightarrow\C} \) factors through \( \R \).
\begin{lem}\label{lem:UniqueRealSubbundle}
 Let \( (\vb{E},\epsilon) \) be a symmetric complex vector bundle. There exists a unique real subbundle \( \Re(\vb{E},\epsilon)\subset\vb{E} \) such that \( \Re(\vb{E},\epsilon)\otimes_{\R}\C=\vb{E} \) and such that the restriction of \( \epsilon \) to \( \Re(\vb{E},\epsilon) \) is real and positive definite. Concretely, a fibre of \( \Re(\vb{E},\epsilon) \) is given by the real span of any orthonormal basis of the corresponding fibre of \( \vb{E} \).
\end{lem}
\begin{cor}\label{cor:ComplexSymmetric_is_Real}
For any CW complex \( X \),
the monoid of isomorphism classes of real vector bundles over \( X \) is isomorphic to the monoid of isometry classes of symmetric complex vector bundles over \( X \) (with respect to the operations \( \oplus \) and \( \perp \), respectively).
\end{cor}
\begin{proof}[Proof of Lemma~\ref{lem:UniqueRealSubbundle}]
In the case of a vector bundle over a point we may assume without loss of generality that
\begin{equation*}
 (\vb{E},\epsilon)=\big(\C^r,\mm{ 1 &  & 0 \\  & \raisebox{0pt}[10pt]{\( \ddots \)} &   \\ 0 &  & 1}\big)
\end{equation*}
Clearly, the subspace \( \R^r\subset\C^r \) has the required properties. Uniqueness follows from elementary linear algebra.
If \( (\vb{E},\epsilon) \) is an arbitrary symmetric complex vector bundle over a space \( X \), then any point of \( X \) has some neighbourhood over which \( (\vb{E},\epsilon) \) can be trivialized in the form above. We know how to define \( \Re(\vb{E},\epsilon) \) over each such neighbourhood, and by uniqueness these local bundles can be glued together.
\end{proof}

\begin{proof}[Proof of Corollary~\ref{cor:ComplexSymmetric_is_Real}]
A map in one direction is given by sending a symmetric complex vector bundle \( (\vb{E},\epsilon) \) to \( \Re(\vb{E},\epsilon) \). Conversely, with any real vector bundle \( \vb{E} \) over \( X \) we may associate a symmetric complex vector bundle \( (\vb{E}\otimes_{\R}\C,\sigma_{\C}) \), where \( \sigma_{\C} \) is the \( \C \)-linear extension of some inner product \( \sigma \) on \( \vb{E} \). Since \( \sigma \) is defined uniquely up to isometry, so is \( (\vb{E}\otimes_{\R}\C,\sigma_{\C}) \). See \cite{MilnorHusemoller}*{Chapter~V, \S~2} for a proof that avoids the uniqueness part of the preceding lemma.
\end{proof}

\paragraph{Representing topological K-groups.}
A standard construction of the cohomology theories associated with \( \K^0 \) and \( \KO^0 \) is based on the fact that these functors are representable in the homotopy category \( \uH \) of topological spaces. The starting point is the homotopy classification of vector bundles: Let us write \( \Gr_{r,n} \) for the Grassmannian \( \Gr(r,\C^{r+n}) \) of complex \( r \)-bundles in \( \C^{r+n} \), and let \( \Gr_r \) be the union of \( \Gr_{r,n}\subset\Gr_{r,n+1}\subset\cdots \) under the obvious inclusions. Denote by \( \vb{U}_{r,n} \) and \( \vb{U}_r \) the universal \( r \)-bundles over these spaces. For any connected paracompact Hausdorff space \( X \) we have a one-to-one correspondence between the set \( \mathrm{Vect}_r(X) \) of isomorphism classes of rank \( r \) complex vector bundles over \( X \) and homotopy classes of maps from \( X \) to \( \Gr_r \): a homotopy class \( [f] \) in \( \uH(X,\Gr_{r}) \) corresponds to the pullback of \( \vb{U}_r \) along \( f \) \cite{Husemoller:FibreBundles}*{Chapter~3, Theorem~7.2}.

To describe \( \K^0(X) \), we need to pass to \( \Gr \), the union of the \( \Gr_r \) under the embeddings \( \Gr_r\hookrightarrow \Gr_{r+1} \) that send a complex \( r \)-plane \( W \) to \( \C\oplus W \).
\begin{thm}\label{thm:Unstable_Representability_of_K} For finite-dimensional CW complexes \( X \) we have natural isomorphisms
 \begin{equation}\label{eq:K0-unstable}
  \K^0(X)\cong\uH(X,\Z\times{\Gr})
 \end{equation}
 such that, for \( X=\Gr_{r,n} \), the class \( [\vb{U}_{r,n}]+(d-r)[\C] \) in \( \K^0(\Gr_{r,n}) \) corresponds to the inclusion \( \Gr_{r,n}\hookrightarrow\{d\}\times\Gr_{r,n}\hookrightarrow\Z\times\Gr \).
\end{thm}
\begin{proof}
The theorem is of course well-known, see for example \cite{Adams}*{page 204}. To deduce it from the homotopy classification of vector bundles, we note first that any CW complex is paracompact and Hausdorff \cite{Hatcher:VBKT}*{Proposition~1.20}. Moreover, we may assume that \( X \) is connected. The product \( \Z\times\Gr \) can be viewed as the colimit of the inductive system
\begin{equation*}
 \coprod_{d\geq 0} \{d\}\times \Gr_d \hookrightarrow \coprod_{d\geq -1 } \{d\}\times \Gr_{d+1}
                                   \hookrightarrow \coprod_{d\geq -2 } \{d\}\times \Gr_{d+2}
                                   \hookrightarrow \dots \subset \Z\times\Gr
\end{equation*}
Any continuous map from \( X \) to \( \Z\times\Gr \) factors though one of the components \( \colim_n(\{d\}\times\Gr_n) \). By cellular approximation, it is  in fact homotopic to a map that factors through \( \{d\}\times\Gr_n \) for some \( n \). Thus,
\begin{equation*}
\uH(X,\Z\times\Gr)\cong\coprod_{d\in\Z}\colim_n\mathrm{Vect}_n(X)
\end{equation*}
where the colimit is taken over the maps \( \mathrm{Vect}_n(X)\rightarrow\mathrm{Vect}_{n+1}(X) \) sending a vector bundle \( \vb{E} \) to \( \C\oplus\vb{E} \). We define a map from the coproduct to \( \K^0(X) \) by sending a vector bundle \( \vb{E} \) in the \( d^{\text{th}} \) component to the class \( [\vb{E}]+(d-\rank\vb{E})[\C] \) in \( \K^0(X) \). To see that this is a bijection, we use the fact that every vector bundle \( \vb{E} \) over a finite-dimensional CW complex has a stable inverse: a vector bundle \( \vb{E}^{\perp} \) over \( X \) such that \( {\vb{E}\oplus\vb{E}^{\perp}} \) is a trivial bundle \cite{Husemoller:FibreBundles}*{Chapter~3, Proposition~5.8}.
\end{proof}

If we replace the complex Grassmannians by real Grassmannians \( {\R}{\Gr_{r,n}} \), we obtain the analogous statement that \( \KO^0 \) can be represented by \( \Z\times\R{\Gr} \). Equivalently, but more in the spirit of Definition~\ref{def:K_and_KO-concrete}, we could work with the following spaces:
\begin{defn}
Let \( (V,\nu) \) be a symmetric complex vector space, and let \( \Gr(r,V) \) be the Grassmannian of complex \( k \)-planes in \( V \). The ``non-degenerate Grassmannian''
\begin{equation*}
 \Grnd(r,(V,\nu))
\end{equation*}
is the open subspace of \( \Gr(r,V) \) given by \( r \)-planes \( T \) for which the restriction \( \nu|_T \) is non-degenerate.
\end{defn}
Complexification induces an inclusion of \( {\R}{\Gr(k,\Re(V,\nu))} \) into \( \Grnd(r,(V,\nu)) \), which, by Lemma~\ref{lem:Grnd_vs_RGr} below, is a homotopy equivalence. So let \( \Grnd_{r,n} \) abbreviate \( \Grnd(r,\HH^{r+n}) \), where \( \HH \) is the hyperbolic plane \( (\C^2,\mm{0 & 1 \\ 1 & 0}) \), and let \( \vb{U}_{r,n}^{\text{nd}} \) denote the restriction of the universal bundle over \( \Gr(r,\C^{2r+2n}) \) to \( \Grnd_{r,n} \).
Then colimits \( \Grnd_r \) and \( \Grnd \) can be defined in the same way as for the usual Grassmannians, and, for finite-dimensional CW complexes \( X \), we obtain natural isomorphisms
\begin{align}\label{eq:KO0_unstable}
 \KO^0(X)&\cong\uH(X,\Z\times{\Grnd})
\end{align}
Here, for even \( (d-r) \), the inclusion \( \Grnd_{r,n}\hookrightarrow\{d\}\times\Grnd_{r,n}\hookrightarrow\Z\times\Grnd \) corresponds to the class of \( [\vb{U}_{r,n}^{\text{nd}}]+\frac{d-r}{2}[\HH] \) in \( \GW^0(\Grnd_{r,n}) \).

\begin{lem}\label{lem:Grnd_vs_RGr}
 For any symmetric complex vector space \( (V,\nu) \), the following inclusion is a homotopy equivalence:
 \begin{align*}
  \R{\Gr(k,\Re(V,\nu))} &\overset{j}{\hookrightarrow} \Grnd(k,(V,\nu))\\
  U                     &\mapsto                      U\otimes_\R\C
 \end{align*}
\end{lem}
\begin{proof}
Consider the projection \( \pi\colon{V=\Re(V,\nu)\oplus i\Re(V,\nu)\twoheadrightarrow\Re(V,\nu)} \). We define a retract \( r \) of \( j \) by sending a complex \( k \)-plane \( T\in\Grnd(k,(V,\nu)) \) to
\(  \pi(\Re(T,\nu|_T)) \subset\Re(V,\nu) \). This is indeed a linear subspace of real dimension \( k \): since \( \nu \) is positive definite on \( \Re(T,\nu|_T) \) but negative definite on \( i\Re(V,\nu) \), the intersection \( \Re(T,\nu|_T)\cap i\Re(V,\nu) \) is trivial.

More generally, we can define a family of endomorphisms of \( V \) parametrized by \( t\in[0,1] \) by
\begin{align*}
 \pi_t\colon{\Re(V,\nu)\oplus i\Re(V,\nu)} & \twoheadrightarrow \Re(V,\nu)\oplus i\Re(V,\nu) \\
                (x,y)                       & \mapsto            (x,ty)
\end{align*}
This family interpolates between the identity \( \pi_1 \) and the projection \( \pi_0 \), which we can identify with \( \pi \). We claim that \begin{equation*}
 \pi_t(\Re(T,\nu|_T)) \subset V
\end{equation*}
is a real linear subspace of dimension \( k \) on which \( \nu \) is real and positive definite. The claim concerning the dimension has already been verified in the case \( t=0 \) and follows for non-zero \( t \) from the fact that \( \pi_t \) is an isomorphism. Now take a non-zero vector \( v\in\pi_t(\Re(T,\nu|_T)) \) and write it as \( v=x+tiy \), where \( x,y\in\Re(V,\nu) \) and \( x+iy\in\Re(T,\nu|_T) \). Since \( \nu(x,x) \), \( \nu(y,y) \) and \( \nu(x+iy,x+iy) \) are all real we deduce that \( \nu(x,y)=0 \); it follows that \( \nu(v,v) \) is real as well. Moreover, since \( \nu(x+iy,x+iy) \) is positive we have \( \nu(x,x)>\nu(y,y) \), so that \( \nu(v,v)> (1-t^2)\nu(y,y) \). In particular, \( \nu(v,v)>0 \) for all \( t\in[0,1] \), as claimed.

It follows that \( T\mapsto\pi_t(\Re(T,\nu|_T))\otimes_\R\C \) defines a homotopy from \( j\circ r \) to the identity on \( \Grnd(k,(V,\nu)) \).
\end{proof}

\paragraph{K-spectra and cohomology theories.}
The infinite Grassmannian \( \Gr \) can be identified with the classifying space \( \BU \) of the infinite unitary group. Consequently, \( \K^0 \) can be represented by \( \Z\times{\BU} \), which by Bott periodicity is equivalent to its own two-fold loopspace \( \Omega^2(\Z\times{\BU}) \). This can be used to construct a \( 2 \)-periodic \( \Omega \)-spectrum \( \SPtK \) in the stable homotopy category \( \SH \) whose even terms are all given by \( {\Z\times\BU} \). Similarly, \( \R{\Gr} \) is equivalent to the classifying space \( \BO \) of the infinite orthogonal group, and Bott periodicity in this case says that \( \Z\times\BO \) is equivalent to \( \Omega^8(\Z\times\BO) \).
Thus, one obtains a spectrum \( \SPtKO \) in \( \SH \) which is \( 8 \)-periodic. The associated cohomology theories are given by
\begin{align*}
 \K^i(X)  &:=\SH(\SgmInf(X_+),S^i\smsh\SPtK)\\
 \KO^i(X) &:=\SH(\SgmInf(X_+),S^i\smsh\SPtKO)
\end{align*}
where \( X_+ \) denotes the union of \( X \) and a disjoint base point, and \( \SgmInf \) is the functor assigning to a pointed space its suspension spectrum. We refer the reader to \cite{Adams}*{III.2} for background and details.

For convenience and later reference, we include here the values of the theories on a point. Since we are in fact dealing with multiplicative theories, these can be summarized in the form of coefficient rings:
\begin{align}
 \K^*(\point)&=\Z\big[g,g^{-1}\big]
 \label{eq:K-coefficients}\\
 \KO^*(\point)&=\quotient{\Z\big[\eta, \alpha, \lambda, \lambda^{-1}\big]}{(2\eta,\;\eta^3,\;\eta\alpha,\;\alpha^2-4\lambda)}
 \label{eq:KO-coefficients}
\end{align}
where \( g \) is of degree \( -2 \) and \( \eta \), \( \alpha \) and \( \lambda \) have degrees \( -1 \), \( -4 \) and \( -8 \), respectively \cite{Bott:K}*{pages~66--74\footnote{The multiplicative relations among the generators are given on page~74, but unfortunately the relation \( \eta\alpha=0 \) is missing. This omission seems to have pervaded much of the literature, and I am indebted to Ian Grojnowski for pointing out the same mistake in an earlier version of this paper. Of course, the relation follows from the fact that \( \KO^{-5}(\point)=0 \).}}.

\subsection{Comparison}
Now suppose \( X \) is a smooth complex variety. We write \( X(\C) \) for the set of complex points of \( X \) equipped with the analytic topology. If \( \vb{E} \) is a vector bundle over \( X \) then \( \vb{E}(\C) \) has the structure of a complex vector bundle over \( X(\C) \), so that we obtain natural maps
\begin{align}
 \K_0(X)  \rightarrow &\K^0(X(\C)) \label{eq:naive_comparison_K}\\
 \GW^0(X) \rightarrow &\KO^0(X(\C))\label{eq:naive_comparison_KO}
\end{align}
and an induced map
\begin{align}
\W^0(X) \rightarrow  &\frac{\KO^0(X(\C))}{\K^0(X(\C))}
\end{align}

We now wish to extend these maps to be defined on \( \GW^i(X) \) and \( \W^i(X) \) for arbitrary \( i \), and also on groups with support and twisted groups. Let us comment on some ``elementary'' constructions that are possible before outlining the approach that we will ultimately follow here.

Firstly, one way to extend the maps to the groups \( \GW^i(X) \) and \( \W^i(X) \) is to use the multiplicative structure of the theories together with Walter's results on projective bundles \cite{Walter:PB}. Namely, for any variety \( X \) one has isomorphisms
\begin{align*}
  \GW^i(X\times\P^1)         &\cong \GW^i(X)\oplus\GW^{i-1}(X)\\
  \KO^{2i}(X(\C)\times S^2)  &\cong \KO^{2i}(X(\C))\oplus\KO^{2i-2}(X(\C))
\end{align*}
This allows an inductive definition of comparison maps, at least for all negative \( i \). Basic properties of these maps, for example compatibility with the periodicities of Grothendieck-Witt and KO-groups, can be checked by direct calculations.

It is less clear how to obtain maps on Witt groups with restricted supports. One possibility, pursued in \cite{Me:RSK-Essay}, is to work on the level of complexes of vector bundles and adapt a construction of classes in relative K-groups described in \cite{Segal} to the case of KO-theory. However, it remains unclear to the author how to see in this approach that the resulting maps are compatible with the boundary morphisms in localization sequences.

Theorem~\ref{thm:mainthm} below could in fact be proved without knowing that the comparison maps respect the boundary morphisms in localization sequences in general. However, \( \A^1 \)-homotopy theory provides an alternative construction of a comparison map for which this property immediately follows from the construction, and which in any case is so compellingly elegant that it would be difficult to argue in favour of any other approach.

\subsection{$\mathbb{A}^1$-homotopy theory}
Theorem~\ref{thm:Unstable_Representability_of_K}
describing \( \K^0 \) in terms of homotopy classes of maps to Grassmannians has an analogue in algebraic geometry, in the context of \( \A^1 \)-homotopy theory. Developed mainly by Morel and Voevodsky, the theory provides a general framework for a homotopy theory of schemes emulating the situation for topological spaces. The authoritative reference is \cite{MorelVoevodsky}; closely related texts by the same authors are \cite{Voevodsky:ICM}, \cite{Morel:Asterisque} and \cite{Morel:Introduction}. See \cite{Nordfjordeid} for a textbook introduction and \cite{Dugger:UHT} for an enlightening perspective on one of the main ideas.
\newcommand{\Spaces}{\mathrm{Spc}}
\newcommand{\Smooth}{\mathrm{Sm}}

We summarize the main points relevant for us in just a few sentences. The category \( \Smooth_k \) of smooth schemes over a field \( k \) can be embedded into some larger category of ``spaces'' \( \Spaces_k \) which is closed under small limits and colimits, and which can be equipped with a model structure. The \( \A^1 \)-homotopy category \( \uH(k) \) over \( k \) is the homotopy category associated with this model category.

In fact, there are several possible choices for \( \Spaces_k \) and many possible model structures yielding the same homotopy category \( \uH(k) \). One possibility is to consider the category of simplicial presheaves over \( \Smooth_k \), or the category of simplicial sheaves with respect to the Nisnevich topology. Both categories contain \( \Smooth_k \) as full subcategories via the Yoneda embedding, and they also contain simplicial sets viewed as constant (pre)sheaves. One may thus apply a general recipe for equipping the category of simplicial (pre)sheaves over a site with a model structure (see \cite{Jardine:simplicial_presheaves}). In a crucial last step, one forces the affine line \( \A^1 \) to become contractible by localizing with respect to the set of all projections \( \A^1\times X\twoheadrightarrow X \).

As in topology, we also have a pointed version \( \pH(k) \) of \( \uH(k) \). Remarkably, these categories contain several distinct ``circles'': the simplicial circle \( S^1 \), the ``Tate circle'' \( \G_m=\A^1-0 \) (pointed at 1) and the projective line \( \P^1 \) (pointed at \( \infty \)). These are related by the intriguing formula \( \P^1=S^1\smsh\G_m \).
A common notational convention which we will follow is to define
\begin{equation*}
 S^{p,q}:=(S^1)^{\smsh(p-q)}\smsh\G_m^{\smsh q}
\end{equation*}
for any \( p\geq q \). In particular, we then have \( S^1=S^{1,0} \), \( \G_m=S^{1,1} \) and \( \P^1=S^{2,1} \).

One can take the theory one step further by passing to the stable homotopy category \( \SH(k) \), a triangulated category in which the suspension functors \( {S^{p,q}\smsh-} \) become invertible. This category is usually constructed using \( \P^1 \)-spectra. The triangulated shift functor is given by suspension with the simplicial sphere \( S^{1,0} \).

Finally and crucially, the analogy with topology can be made precise: when we take our ground field \( k \) to be the complex numbers, or more general any subfield of \( \C \), we have a complex realization functor
\begin{align}
 \uH(k)&\rightarrow \uH \label{eq:complex_realization}
\end{align}
that sends a smooth scheme \( X \) to its set of complex points \( X(\C) \) equipped with the analytic topology. There is also a pointed realization functor and, moreover, a triangulated functor of the stable homotopy categories
\begin{align}
 \SH(k) &\rightarrow \SH \label{eq:complex_realization_stable}
\end{align}
which takes \( \SgmInf(X_+) \) to \( \SgmInf(X(\C)_+) \) for any smooth scheme \( X \) \citelist{\cite{Riou:Thesis}*{Th{\'e}o\-r{\`e}me~I.123}\cite{Riou:CatHomStable}*{Th{\'e}o\-r{\`e}me~5.26}}.

\subsection{Representing algebraic and hermitian K-theory}\label{sec:Representing_algK}
Grassmannians of \( r \)-planes in \( k^{n+r} \) can be constructed as smooth projective varieties over any field \( k \). Viewing them as objects in \( \Spaces_k \), we can form their colimits \( \Gr_r \) and \( \Gr \) in the same way as in topology.
The following analogue of Theorem~\ref{thm:Unstable_Representability_of_K} is established in \cite{MorelVoevodsky}*{\S~4}; see Th{\'e}or{\`e}me~III.3 and Assertion~III.4 in \cite{Riou:Thesis}.

\begin{thm}
 For smooth schemes \( X \) over \( k \) we have natural isomorphisms
 \begin{align}\label{eq:K0_unstable}
  \K_0(X) &\cong\uH(k)(X,\Z\times{\Gr})
  \end{align}
 such that the inclusion of \( \Gr_{r,n}\hookrightarrow\{d\}\times\Gr_{r,n}\hookrightarrow\Z\times\Gr \) corresponds to the class \( [\vb{U}_{r,n}]+(d-r)[\OO] \) in \( \K_0(\Gr_{r,n}) \).
\end{thm}

An analogous result for hermitian K-theory has recently been obtained by Schlichting and Tripathi\footnote{Talk ``Geometric representation of hermitian K-theory in \( \mathbb{A}^1 \)-homotopy theory'' at the Workshop ``Geometric Aspects of Motivic Homotopy Theory'', 6.--10. September 2010 at the Hausdorff Center for Mathematics, Bonn}: Let \( \Grnd_{r,n} \) denote the ``non-degenerate Grassmannians'' defined as open subvarieties of \( \Gr_{r,r+2n} \) as above, and let \( \Grnd_r \) and \( \Grnd \) be the respective colimits. Then for smooth schemes over \( k \) we have natural isomorphisms
\begin{align}\label{eq:GW0_unstable}
  \GW^0(X)&\cong\uH(k)(X,\Z\times{\Grnd})
\end{align}
It follows from the construction that, when \( (d-r) \) is even, the inclusion of \( \Grnd_{r,n}\hookrightarrow\{d\}\times\Grnd_{r,n}\hookrightarrow\Z\times\Grnd \) corresponds to the class of \( [\vb{U}_{r,n}^{\text{nd}}]+\frac{d-r}{2}[\HH] \) in \( \GW^0(\Grnd_{r,n}) \), where \( \vb{U}_{r,n}^{\text{nd}} \) is the universal symmetric bundle over \( \Grnd_{r,n} \).

The fact that hermitian K-theory is representable in \( \uH(k) \) has been known for longer, see \cite{Hornbostel}. One of the advantages of having a geometric description of a representing space, however, is that one can easily see what its complex realization is. In particular, this gives us an alternative way to define the comparison maps. For any smooth complex scheme \( X \) we have the following commutative squares, in which the left vertical arrows are the comparison maps \eqref{eq:naive_comparison_K} and \eqref{eq:naive_comparison_KO}, the right vertical arrows are induced by the complex realization functor \eqref{eq:complex_realization}.
\begin{align*}
 &
 \xymatrix@C=3pt{
  {\K_0(X)}\ar@{}[r]|-{\cong} \ar[d] & {\uH(\C)(X,\Z \times \Gr)} \ar[d]\\
  {\K^0(X(\C))}\ar@{}[r]|-{\cong}    & {\uH(X(\C),\Z \times \Gr)}
 }
 &&&
 \xymatrix@C=3pt{
  {\GW^0(X)}\ar@{}[r]|-{\cong} \ar[d] & {\uH(\C)(X,\Z \times \Grnd)} \ar[d]\\
  {\KO^0(X(\C))}\ar@{}[r]|-{\cong}    & {\uH(X(\C),\Z \times \Grnd)}
 }
\end{align*}
Some of the results quoted here are in fact known in a much greater generality. Firstly, higher algebraic and hermitian K-groups of \( X \) are obtained by passing to suspensions of \( X \) in \eqref{eq:K0_unstable} and \eqref{eq:GW0_unstable}. Even better, algebraic and hermitian K-theory are representable in the stable \( \A^1 \)-homotopy category \( \SH(k) \). Let us make the statement a little more precise by fixing some notation. Given a spectrum \( \mathbb{E} \) in \( \SH(k) \), we obtain a bigraded reduced cohomology theory \( \rE^{*,*} \) on \( \pH(k) \) and a corresponding unreduced theory \( E^{*,*} \) on \( \uH(k) \) by setting
\begin{align*}
 \rE^{p,q}(\mathcal{X}) &:=\SH(k)(\SgmInf\mathcal{X}, S^{p,q}\smsh\mathbb{E}) && \text{ for \( \mathcal{X}\in\pH(k) \)} \\
  E^{p,q}(X)            &:= \rE^{p,q}(X_+)                                    && \text{ for \( X\in\uH(k) \)}
\end{align*}
A spectrum \( \SPaK \) representing algebraic K-theory was first constructed in \cite{Voevodsky:ICM}*{\S~6.2}; see \cite{Riou:Thesis} or \cite{Riou:K} for some further discussion. It is (2,1)-periodic, meaning that in \( \SH(k) \) we have an isomorphism
\begin{equation*}
 S^{2,1}\smsh\SPaK\overset{\cong}\rightarrow\SPaK
\end{equation*}
Thus, the bigrading of the corresponding cohomology theory \( \K^{p,q} \) is slightly artificial. The identification with the usual notation for algebraic K-theory is given by
\begin{align}\label{eq:motivic_K-indices}
        \K^{p,q}(X) &= \K_{2q-p}(X)
\end{align}
For hermitian K-theory we have an (8,4)-periodic spectrum \( \SPaKO \), and the corresponding cohomology groups \( \Kh^{p,q} \) are honestly bigraded. The translation into the notation used for hermitian K-groups in Section~\ref{subsec:W} is given by
\begin{align}\label{eq:Representable_vs_Schlichting}
        \Kh^{p,q}(X) &= \GW^q_{2q-p}(X)
\end{align}
We will refer to the number \( 2q-p \) as the degree of the group \( \Kh^{p,q}(X) \). The relation with Balmer's Witt groups obtained by combining \eqref{eq:Representable_vs_Schlichting} and \eqref{eq:Schlichting_vs_BW} is illustrated by the following table:\vspace{-0.5\baselineskip}
\begin{center}
\newcommand{\bs}{\boldsymbol}
\begin{tabular}[t]{M|MMMMMMMMM}
\toprule
\bs{\Kh^{p,q}} & p=0        & 1         & 2          & 3        & 4           & 5         & 6          & 7     \\
\midrule
q=0            & \bs{\GW^0} & \bs{\W^1} & \bs{\W^2}  & \bs{\W^3} & \bs{\W^0}  & \bs{\W^1} & \bs{\W^2}  & \bs{\W^3}  \\
q=1            & \GW^1_2    & \GW^1_1   & \bs{\GW^1} & \bs{\W^2} & \bs{\W^3}  & \bs{\W^0} & \bs{\W^1}  & \bs{\W^2}  \\
q=2            & \GW^2_4    & \GW^2_3   & \GW^2_2    & \GW^2_1   & \bs{\GW^2} & \bs{\W^3} & \bs{\W^0}  & \bs{\W^1}  \\
q=3            & \GW^3_6    & \GW^3_5   & \GW^3_4    & \GW^3_3   & \GW^3_2    & \GW^3_1   & \bs{\GW^3} & \bs{\W^0}  \\
\bottomrule
\end{tabular}
\end{center}

\vspace{0.5\baselineskip}
As for the representing spaces in the unstable homotopy category, it is known that the complex realizations of \( \SPaKO \) and \( \SPaK \) represent real and complex topological K-theory. This is well-documented in the latter case, see for example \cite{Riou:Thesis}*{Proposition~VI.12}. For \( \SPaKO \), our references are slightly thin. Since the emphasis in this article is on showing how such a result in \( \A^1 \)-homotopy theory can be used for some concrete computations, we will at this point succumb to an ``axiomatic'' approach --- the key statements we will be using are as follows:

\begin{myAxioms}\label{myAxioms}
There exist spectra \( \SPaK \) and \( \SPaKO \) in \( \SH(\C) \) representing algebraic K-theory and hermitian K-theory in the sense described above, such that:
\begin{enumerate}[label=(\alph*)]
 \item The complex realization functor \eqref{eq:complex_realization_stable} takes \( \SPaK \) to \( \SPtK \) and \( \SPaKO \) to \( \SPtKO \).
 \item We have an exact triangle in \( \SH(\C) \) of the form
       \begin{align}\label{eq:KKOtriang}
        \SPaKO \smsh S^{1,1}\overset{\eta}{\rightarrow}\SPaKO\rightarrow\SPaK\rightarrow S^{1,0}\smsh\dots
       \end{align}
       which corresponds to the usual triangle in \( \SH \).
\end{enumerate}
\end{myAxioms}

These results are announced in \cite{Morel}. Independent constructions of spectra representing hermitian K-theory can be found in \cite{Hornbostel} and in a recent preprint of Panin and Walter \cite{PaninWalter:BO}.

\section{The comparison maps}\label{sec:comparison}
It follows immediately from \ref{myAxioms} that complex realization induces comparison maps
\begin{alignat*}{2}
        &\phantom{:=::}\tilde{k}^{p,q}\colon{}   & \rK^{p,q}(\mathcal{X})&\rightarrow\rK^p(\mathcal{X}(\C)) \\
        &\phantom{:=::}\tilde{k}_h^{p,q}\colon{} & \rKh^{p,q}(\mathcal{X})&\rightarrow\rKO^p(\mathcal{X}(\C))
\intertext{%
and hence comparison maps \( k^{p,q} \) and \( k_h^{p,q} \) on K- and hermitian K-groups. In particular, in degrees \( 0 \) and \( -1 \) we have maps}
     &\phantom{:=::}k^{0,0}\colon{}&\K_0(X)&\rightarrow \K^0(X(\C))\\
        \gw^q&:=k_h^{2q,q}\colon{}&\GW^q(X)&\rightarrow \KO^{2q}(X(\C))\\
        \w^q&:=k_h^{2q-1,q-1}\colon{}&\W^q(X)&\rightarrow \KO^{2q-1}(X(\C))
\end{alignat*}
for any smooth complex scheme \( X \). Some good properties of these maps follow directly from the construction:
\begin{itemize}
\item They commute with pullbacks along morphisms of smooth schemes.
\item They are compatible with suspension isomorphisms.
\item They are compatible with the periodicity isomorphisms, so we can identify \( k_h^{p,q} \) with \( k_h^{p+8,q+4} \) (and hence \( \w^{q} \) with \( \w^{q+4} \) and \( \gw^q \) with \( \gw^{q+4} \)).
\end{itemize}
It is also clear that they are compatible with long exact sequences arising from exact triangles in \( \SH(\C) \). This will be particularly useful in the following two cases.

\paragraph{Localization sequences.}
Given a smooth closed subscheme \( Z \) of a smooth scheme \( X \), we have an exact triangle
\begin{equation*}
        \SgmInf(X-Z)_+\rightarrow \SgmInf X_+ \rightarrow \SgmInf{\left(\tfrac{X}{X-Z}\right)}\rightarrow S^{1,0}\smsh\dots
\end{equation*}
in \( \SH(\C) \). It induces long exact ``localization sequences'' for cohomology theories. For example, for hermitian K-theory we obtain sequences of the form
\begin{equation}\label{seq:Kh-localization}
   \dots \rightarrow \rKh^{p,q}\left(\tfrac{X}{X-Z}\right) \rightarrow \Kh^{p,q}(X) \rightarrow \Kh^{p,q}(X-Z) \rightarrow \rKh^{p+1,q}\left(\tfrac{X}{X-Z}\right) \rightarrow \dots
\end{equation}
The comparison maps commute with all maps appearing in this sequence and the corresponding sequence of topological KO-groups.

The space \(X/(X-Z)\) depends only on the normal bundle \(\N\) of \(Z\) in \(X\). To make this precise, we introduce the Thom space of a vector bundle \(\vb E\) over an arbitrary smooth scheme \(Z\), defined as the homotopy quotient of \(\vb E\) by the complement of the zero section:
\begin{equation*}
 \Thom_Z(\vb E) := \quotient{\vb E}{(\vb E-Z)}
\end{equation*}
Using a geometric construction known as deformation to the normal bundle, Morel and Voevodsky show in Theorem~2.23 of \cite{MorelVoevodsky}*{Chapter~3} that \(X/(X-Z)\) is canonically isomorphic to \(\Thom_Z(\N)\) in the unstable pointed \(\A^1\)-homotopy category. Thus, sequence \eqref{seq:Kh-localization} can be rewritten in the following form:
\begin{equation*}
\begin{aligned}
   {\dots} \rightarrow
   {\rKh^{p,q}(\Thom_Z{\N})}    & \rightarrow
   {\Kh^{p,q}(X)}               \rightarrow
   {\Kh^{p,q}(X-Z)}             \rightarrow
   \qquad \\
   \qquad
   {\rKh^{p+1,q}(\Thom_Z{\N})}  & \rightarrow
   {\Kh^{p+1,q}(X)}             \rightarrow
   {\Kh^{p+1,q}(X-Z)} \rightarrow
   {\dots}
\end{aligned}
\end{equation*}

\paragraph{Karoubi/Bott sequences.}
The KO- and K-groups of a topological space \( X \) fit into a long exact sequence known as the Bott sequence \citelist{\cite{Bott:K}*{pages~75 and 112\footnote{Unfortunately, there are misprints on both pages. In particular, the central group in the diagram on page 112 should be \( \K^0 \).}}\cite{BrunerGreenlees}*{4.I.B}}.
It has the form
\begin{align}\label{seq:Bott}
 \begin{aligned}
        {\dots}&\rightarrow{\KO^{2i-1}\!{X}}\rightarrow{\KO^{2i-2}\!{X}}\rightarrow
                {\K^0\!{X}}\rightarrow{\KO^{2i}\!{X}}\rightarrow{\KO^{2i-1}\!{X}}\rightarrow
                {\K^1\!{X}\quad}\\
        &{\quad}\rightarrow{\KO^{2i+1}\!{X}}\rightarrow{\KO^{2i}\!{X}}\rightarrow
                {\K^0\!{X}}\rightarrow{\KO^{2i+2}\!{X}}\rightarrow{\KO^{2i+1}\!{X}}\rightarrow
                {\dots}
 \end{aligned}
\end{align}
The maps from KO- to K-groups are essentially given by complexification (or, depending on our choice of description of KO-groups, by forgetting the symmetric structure of a complex symmetric bundle), and the maps from K- to KO-groups are given by sending a complex vector bundle to its underlying real bundle (or to the associated hyperbolic bundle). The maps between KO-groups are given by multiplication with the generator \( \eta \) of \( \KO^{-1}(\point) \) (see \eqref{eq:KO-coefficients}).

This long exact sequence is induced by the triangle described in \ref{myAxioms}. The sequence arising from the corresponding triangle \eqref{eq:KKOtriang} in the stable \( \A^1 \)-homotopy category is known as the Karoubi sequence. The comparison maps induce a commutative ladder diagram that allows us to compare the two. Near degree zero, this takes the following form:
\begin{equation}\label{seq:Karoubi-comparison}
\xymatrix@C=11pt{
    {\dots}        \ar[r]
  & {\Kh^{2i-1,i}\!{X}} \ar[r]\ar[d]
  & {\GW^{i-1}\!{X}} \ar[r]\ar[d]
  & {\K_0\!{X}}      \ar[r]\ar[d]
  & {\GW^i\!{X}}     \ar[r]\ar[d]
  & {\W^i\!{X}}      \ar[r]\ar[d]
  & {0}            \ar[r]\ar[d]
  & {\dots}
  \\
    {\dots}        \ar[r]
  & {\KO^{2i-1}\!{X}}\ar[r]
  & {\KO^{2i-2}\!{X}}\ar[r]
  & {\K^0\!{X}}      \ar[r]
  & {\KO^{2i}\!{X}}  \ar[r]
  & {\KO^{2i-1}\!{X}}\ar[r]
  & {\K^1\!{X}} \ar[r]
  & {\dots}
  }
\end{equation}
As a consequence, the comparison maps \( \w^i \) factor as
\begin{equation*}
 \W^i(X)\rightarrow\frac{\KO^{2i}(X)}{\K^0(X)} \rightarrow \KO^{2i-1}(X)
\end{equation*}
For cellular varieties, or more generally for spaces for which the odd topological K-groups vanish, the second map in this factorization is an isomorphism.

\paragraph{Groups with restricted support.} Comparing the localization sequences \eqref{seq:GW-localization} and \eqref{seq:Kh-localization}, we see that the groups \(\rKh^{p,q}(\frac{X}{X-Z})\) play the role of hermitian K-groups of \(X\) supported on \(Z\). This should be viewed as part of any representability statement, see for example \cite{PaninWalter:BO}*{Theorem~6.5}. Alternatively, a formal identification of the groups in degrees zero and below using only the minimal assumptions we have stated could be achieved as follows:
\begin{lem}\label{lem:Kh-with-support}
Let \(Z\) be a smooth closed subvariety of a smooth quasi-projective variety \(X\).
We have the following isomorphisms:
\begin{align*}
  \rKh^{2q,q}(\tfrac{X}{X-Z})  & \cong \GW^q_Z(X)\\
  \rKh^{p,q}(\tfrac{X}{X-Z})  & \cong \W^{p-q}_Z(X)\;\text{ for \( \;2q-p<0 \)}
\end{align*}
\end{lem}
\begin{proof}
Consider \( Z=Z\times\{0\} \) as a subvariety of \(X\times\A^1\). Its open complement \( (X\times\A^1)-Z \) contains \( X=X\times\{1\} \) as a retract. Thus, the projection from \( X\times\A^1 \) onto \( X \) induces a splitting of the localization sequences associated with \( (X\times\A^1 -Z)\hookrightarrow X\times\A^1\), and we have
\begin{align*}
\GW_Z^{i+1}(X\times\A^1)
   &\cong\coker\left(
\GW^{i+1}_1(X\times\A^1)\hookrightarrow\GW^{i+1}_1(X\times\A^1-Z)
   \right)
   \\
\rKh^{2i+2,i+1}(\tfrac{X\times\A^1}{X\times\A^1-Z})
   &\cong\coker\left(
\Kh^{2i+1,i+1}(X\times\A^1)\hookrightarrow\Kh^{2i+1,i+1}(X\times\A^1-Z)
   \right)
\end{align*}
By \eqref{eq:Representable_vs_Schlichting}, we can identify the groups appearing on the right, so we obtain an induced isomorphism of the cokernels. The quotient \(X\times\A^1/(X\times\A^1-Z)\) can be identified with the suspension of \(X/(X-Z)\) by \(S^{2,1}\), so we have an isomorphism
\begin{align*}
  \rKh^{2i+2,i+1}(\tfrac{X\times\A^1}{X\times\A^1-Z})&\cong\rKh^{2i,i}(\tfrac{X}{X-Z})
\intertext{%
On the other hand, we have analogous isomorphisms
}%
  \GW_Z^{i+1}(X\times\A^1)&\cong\GW^i_Z(X)
\end{align*}
for Grothendieck-Witt and Witt groups. For Witt groups, this is a special case of Theorem~2.5 in \cite{Nenashev:Gysin}, the case when \(Z=X\) being Theorem~8.2 in \cite{BalmerGille:Koszul}. The corresponding isomorphisms of Grothendieck-Witt groups may be deduced via Karoubi induction.
The proof in lower degrees is analogous.
\end{proof}

\subsection{Twisting by line bundles}\label{sec:twists}
As described in Section \ref{subsec:W}, there is a natural notion of Witt groups twisted by line bundles. In the homotopy theoretic approach, such a twist can be encoded by passing to the Thom space of the bundle.
\begin{defn}
   For a vector bundle \( \vb{E} \) of constant rank \( r \) over a smooth scheme \( X \), we define the hermitian K-groups of \( X \) with coefficients in \( \vb{E} \) by
   \begin{align*}
   \Kh^{p,q}(X;\vb{E}) &:= \rKh^{p+2r,q+r}(\Thom{\vb{E}})
   \intertext{Likewise, for any complex vector bundle of rank \( r \) over a topological space \( X \), we define}
   \KO^p(X;\vb{E})     &:=\rKO^{p+2r}(\Thom{\vb{E}})
   \end{align*}
\end{defn}
When \( \vb{E} \) is a trivial bundle, its Thom space is just a suspension of \( X \), so that \( \KO^{p,q}(X;\vb{E}) \) and \( \KO^{p,q}(X) \) agree.
\begin{lem}\label{lem:Kh-with-twist}
For any vector bundle \( \vb{E} \) over a smooth quasi-projective variety \( X \), we have isomorphisms
 \begin{align*}
  \Kh^{2q,q}(X;\vb{E}) &\cong \GW^q(X;\det{\vb{E}})\\
  \Kh^{p,q}(X;\vb{E})  &\cong \W^{p-q}(X;\det{\vb{E}})\;\text{ for \( \;2q-p<0 \)}
 \end{align*}
\end{lem}
\begin{proof}
This follows from Lemma~\ref{lem:Kh-with-support} and Nenashev's Thom isomorphisms for Witt groups: for any vector bundle \( \vb{E} \) of rank \( r \) there is a canonical Thom class in \( \W^r_X(\vb{E}) \) which induces an isomorphism \( \W^i(X;\det\vb{E})\cong\W_X^{i+r}(\vb{E}) \) by multiplication \cite{Nenashev:Gysin}*{Theorem~2.5}. This Thom class actually comes from a class in \( \GW^r_X(\vb{E}) \), and, as in the proof of Lemma~\ref{lem:Kh-with-support}, we can deduce that it induces an analogous isomorphism on Grothendieck-Witt groups via Karoubi induction.
\end{proof}
\begin{rem*}
The isomorphisms of Lemmas~\ref{lem:Kh-with-support} and \ref{lem:Kh-with-twist} are constructed here in a rather ad hoc fashion, and we have taken little care in recording their precise form. Whenever we give an argument concerning the comparison maps on ``twisted groups'' in the following, we do all constructions on the level of representable groups of Thom spaces. The identifications with the usual twisted groups are only needed to identify the final output of concrete calculations as in Section~\ref{sec:examples}.
\end{rem*}
It follows similarly from Thom isomorphisms in topology that the groups \( \KO(X;\vb{E}) \) only depend on the determinant line bundle of \( \vb{E} \):
\begin{lem}\label{lem:KOtwists}
For complex vector bundles \( \vb{E} \) and \( \vb{F} \) on a topological space \( X \) with identical first Chern class modulo 2, we have
 \begin{align*}
  \KO^p(X;\vb{E})\cong\KO^p(X;\vb{F})
 \end{align*}
\end{lem}
\begin{proof}
A complex vector bundle \( \vb{E} \) whose first Chern class vanishes modulo \( 2 \) has a spin structure and is therefore oriented with respect to KO-theory \cite{ABS:Clifford}*{\S~12}. That is, we have a Thom isomorphism
\begin{align*}
 \KO^p{X}\overset{\cong}\longrightarrow\rKO^{p+2r}(\Thom{\vb{E}})
\end{align*}
Now suppose \( c_1(\vb{E})\equiv c_1(\vb{F}) \) mod \( 2 \). We may view \( {\vb{E}\oplus\vb{E}\oplus\vb{F}} \) both as a vector bundle over \( \vb{E} \) and as a vector bundle over \( \vb{F} \), and by assumption it is oriented with respect to KO-theory in both cases. Thus, both groups in the lemma can be identified with \( \KO^p(X;\vb{E}\oplus\vb{E}\oplus\vb{F}) \).
\end{proof}
\begin{rem}
In general the identifications of Lemma~\ref{lem:KOtwists} are non-canonical. Given a spin structure on a real vector bundle, the constructions in \cite{ABS:Clifford} do yield a canonical Thom class, but there may be several different spin structures on the same bundle. Still, canonical identifications exist in many cases. For example, there is a canonical spin structure on the square of any complex line bundle, yielding canonical identifications
\begin{equation*}
 \KO^p(X;\vb{L})\cong\KO^p(X;\vb{L}\otimes\vb{M}^{\otimes 2})
\end{equation*}
for any two complex line bundles \( \vb{L} \) and \( \vb{M} \) over \( X \). Moreover, as different spin structures on a spin bundle over \( X \) are classified by the singular cohomology group \( H^1(X;\Z/2) \), all spin structures arising in the context of complex cellular varieties below will be unique.
\end{rem}

\subsection{The comparison for cellular varieties}
\begin{thm}\label{thm:mainthm}
For a smooth cellular complex variety \( X \), the following comparison maps are isomorphisms:
\begin{align*}
                        \K_0(X)         &\overset{\cong}\longrightarrow \K^0(X(\C)) \\
        \gw^q   \colon{ \GW^q(X)}       &\overset{\cong}\longrightarrow \KO^{2q}(X(\C))\\
        \w^q    \colon{ \W^q(X)}        &\overset{\cong}\longrightarrow \KO^{2q-1}({X(\C)})
\end{align*}
This remains true for twisted groups (see Section~\ref{sec:twists}).
\end{thm}

As indicated in the introduction, the first isomorphism is well-known and almost self-evident, given that both \( \K_0(X) \) and \( \K^0(X(\C)) \) are free abelian of rank equal to the number of cells of \( X \).
In particular, both the algebraic group \( \K_0(\C) \) and the topological K-group \( \K^0(\point) \) are isomorphic to the integers, generated by the trivial line bundle, and the comparison map is the obvious isomorphism.

Let us begin the proof of the theorem by also considering the other two maps first in the case when \( X \) is just a point \( \mathrm{Spec}(\C) \).
We can easily see that the corresponding groups are isomorphic by direct comparison as in Table~\ref{table:point-comparison}.
\begin{table}
\renewcommand{\i}[1]{#1}
\newcommand{\bs}{\boldsymbol}
\begin{center}
\begin{tabular}{M|MMMMMMMMM}
\toprule
\bs{\Kh^{p,q}}(\C) & p=0      & 1      & 2       & 3       & 4        & 5       & 6         & 7      \\
\midrule
\minrowheight{\( \Z \)}
q=0                & \pmb{\Z} & \bs{0} & \i{0}   & \i{0}   & \i{\Z/2} &\i{0}    & \i{0}     & \i{0}  \\
q=1                & \dots    & \dots  & \bs{0}  & \bs{0}  & \i{0}    &\i{\Z/2} & \i{0}     & \i{0}  \\
q=2                & \dots    & \dots  & \dots   & \dots   & \pmb{\Z} & \bs{0}  & \i{\Z/2}  & \i{0}  \\
q=3                & \dots    & \dots  & \dots   & \dots   & \dots    & \dots   & \bs{\pmb{\Z}/2} & \bs{\pmb{\Z}/2}   \\
\midrule
\bs{\KO^p}(\point) & \pmb{\Z} & \bs{0} & \bs{0}  & \bs{0}  & \pmb{\Z} & \bs{0}  & \bs{\pmb{\Z}/2} & \bs{\pmb{\Z}/2}\\
\bottomrule
\end{tabular}
\caption{(Grothendieck-)Witt and KO-groups of a point.}
\label{table:point-comparison}
\end{center}
\end{table}
To see that the isomorphisms are given by our comparison maps, we can use the comparison of the Karoubi and Bott sequences. First, setting \( i=0 \) in Diagram~\eqref{seq:Karoubi-comparison}, we see that \( \gw^0 \) and \( \w^0 \) are isomorphisms on a point.
As \( \W^0(\C) \) is the only non-trivial Witt group of a point, it follows that \( \w^q \) is an isomorphism on a point in general, so that we have
\begin{align*}
\xymatrix{
        {\dots} \ar[r] & {\GW^{q-1}}  \ar[r] \ar[d]^{\gw^{q-1}} & {\K_0} \ar[r] \ar[d]^{\cong} & {\GW^{q}} \ar[r] \ar[d]^{\gw^q} & {\W^q} \ar[r] \ar[d]^{\cong} & {0} \ar[r] \ar[d] & {\dots} \\
        {\dots} \ar[r] & {\KO^{2q-2}} \ar[r]                    & {\K^0} \ar[r]                & {\KO^{2q}}\ar[r]                & {\KO^{2q-1}} \ar[r]          & {0} \ar[r]        & {\dots}
        }
\end{align*}
Given the periodicity of the Grothendieck-Witt groups, repeated applications of the Five Lemma now show that \( \gw^q \) is an isomorphism on a point for all values of \( q \). (This strategy of proof is known as ``Karoubi induction''.)

\medskip
We now treat the hermitian case in general. The case of algebraic/complex K-theory could be dealt with similarly, or deduced from the hermitian case using triangle~\eqref{eq:KKOtriang}. It will be helpful to consider not only the maps \( \gw^q=k_h^{2q,q} \) and \( \w^{q+1}=k_h^{2q+1,q} \) in degrees \( 0 \) and \( -1 \), respectively, but also the maps \( k_h^{2q-1,q} \) in degree \( 1 \) and the maps \( k_h^{2q+2,q} \) in degree \( -2 \). We will prove the following extended statement:
\begin{thm}\label{thm:extendedversion}
For a smooth cellular variety \( X \), the hermitian comparison maps in degrees \( 1 \), \( 0 \), \( -1 \) and \( -2 \) have the properties indicated:
\begin{align*}
        \Kh^{2q-1,q}(X) &\twoheadrightarrow \KO^{2q-1}(X(\C)) \\
        \Kh^{2q,q}(X) &\overset{\cong}{\rightarrow} \KO^{2q}(X(\C)) \\
        \Kh^{2q+1,q}(X) &\overset{\cong}{\rightarrow} \KO^{2q+1}(X(\C))\\
        \Kh^{2q+2,q}(X) &\rightarrowtail \KO^{2q+2}(X(\C))
\end{align*}
The analogous statements for twisted groups are also true.
\end{thm}

\begin{rem}
The map in degree \( 1 \) is not an isomorphism even when \( X \) is a point. For example, it is known that \( \Kh^{-1,0}(\C)=\Z/2 \) (see \cite{KTHEORY-Karoubi}*{Example~18}), from which we may deduce via the Karoubi sequence that \( \Kh^{1,1}(\C)\cong\C^* \). In particular, \( \Kh^{1,1}(\C) \) cannot be isomorphic to \( \KO^1(\point)=0 \).

The map in degree \( -2 \) can be identified with the inclusion of the 2-torsion subgroup of \( \KO^{2q+2}(X(\C)) \) into \( \KO^{2q+2}(X(\C)) \) for any cellular variety \( X \). This follows from the theorem and the description of the KO-groups of cellular varieties given in Lemma~\ref{lem:structure_of_KO}.

In degrees less than \( -2 \), the comparison map is necessarily zero. The problem is that while \( \eta\colon{\W^{p-q}(X)\rightarrow\W^{p-q}(X)} \) is an isomorphism in all negative degrees, the topological \( \eta \) is nilpotent (\( \eta^3=0 \)).
\end{rem}

\medskip
The proof of Theorem~\ref{thm:extendedversion} will proceed by induction over the number of cells of \( X \) and occupy the remainder of this section. To begin the induction, we need to consider the case of only one cell, which immediately reduces to the case of a point by homotopy invariance. In this case, degrees \( 0 \) and \( -1 \) have already been dealt with above. In degrees \( 1 \) and \( -2 \), on the other hand, most of the statements are trivial, and we only need to look at a few particular cases, which we postpone to the end of the proof.

\textbf{Spheres.}
Assuming the theorem to be true for a point, the compatibility of the comparison maps with suspensions immediately shows that the theorem is also true for the reduced cohomology of the spheres \( (\P^{1})^{\wedge d}=S^{2d,d} \). To be precise, the following maps in degrees \( 1 \), \( 0 \), \( -1 \) and \( -2 \) have the properties indicated:
\begin{align*}\label{proof:spheres}
        \rKh^{2q-1,q}(S^{2d,d}) &\twoheadrightarrow \rKO^{2q-1}(S^{2d}) \\
        \rKh^{2q,q}(S^{2d,d}) &\overset{\cong}{\rightarrow} \rKO^{2q}(S^{2d}) \\
        \rKh^{2q+1,q}(S^{2d,d}) &\overset{\cong}{\rightarrow} \rKO^{2q+1}(S^{2d})\\
        \rKh^{2q+2,q}(S^{2d,d}) &\rightarrowtail \rKO^{2q+2}(S^{2d})
\end{align*}

\textbf{Cellular varieties.}
Now let \( X \) be a smooth cellular variety. By definition, \( X \) has a filtration by closed subvarieties \( {\emptyset=Z_0\subset Z_1\subset Z_2 \dots\subset Z_N=X} \) such that the open complement of \( Z_{k} \) in \( Z_{k+1} \) is isomorphic to \( \A^{n_k} \) for some \( n_k \). In general, the subvarieties \( Z_k \) will not be smooth. Their complements \( {U_k:=X-Z_k} \) in \( X \), however, are always smooth as they are open in \( X \). So we obtain an alternative filtration \( X=U_0\supset U_1\supset U_2 \dots \supset U_N=\emptyset \) of \( X \) by smooth open subvarieties \( U_k \). Each \( U_k \) contains a closed cell \( C_{k}\cong\A^{n_k} \) with open complement \( U_{k+1} \).

Our inductive hypothesis is that we have already proved the theorem for \( U_{k+1} \), and we now want to prove it for \( U_k \). We can use the following exact triangle in \( \SH(\C) \):
\begin{align*}
        \SgmInf(U_{k+1})_+ \rightarrow  \SgmInf(U_{k})_+ \rightarrow \SgmInf{\Thom(\N_{C_k\backslash U_k})} \rightarrow S^{1,0}\smsh\dots
\end{align*}

As \( C_k \) is a cell, the Quillen-Suslin theorem tells us that the normal bundle \( \N_{C_k\backslash U_k} \) of \( C_k \) in \( U_k \) has to be trivial. Thus, \( \Thom(\N_{C_k\backslash U_k}) \) is \( \A^1 \)-weakly equivalent to \( S^{2d,d} \), where \( d \) is the codimension of \( C_k \) in \( U_k \). Figure~\ref{fig:proof} displays the comparison between the long exact cohomology sequences induced by this triangle. The inductive step is completed by applying the Five Lemma to each dotted map in the diagram.

\begin{figure}[htbp]
\begin{equation*}
\xymatrix{
                                  \ar[d]\ar@{}[r]_-{\cdots}             & \ar[d]\\
        {\rKh^{2q-1,q}(S^{2d,d})} \ar[d]\ar@{->>}[r]                    & {\rKO^{2q-1}(S^{2d}) }\ar[d]\\
        {\Kh^{2q-1,q}(U_k)      } \ar[d]\ar@{-->}[r]^-{k_h^{2q-1,q}}    & {\KO^{2q-1}(U_k)     }\ar[d]\\
        {\Kh^{2q-1,q}(U_{k+1})  } \ar[d]\ar@{->>}[r]                    & {\KO^{2q-1}(U_{k+1}) }\ar[d]\\
        {\rKh^{2q,q}(S^{2d,d})  } \ar[d]\ar[r]^-{\cong}                 & {\rKO^{2q}(S^{2d})   }\ar[d]\\
        {\Kh^{2q,q}(U_k)        } \ar[d]\ar@{-->}[r]^{k_h^{2q,q}}       & {\KO^{2q}(U_k)       }\ar[d]\\
        {\Kh^{2q,q}(U_{k+1})    } \ar[d]\ar[r]^-{\cong}                 & {\KO^{2q}(U_{k+1})   }\ar[d]\\
        {\rKh^{2q+1,q}(S^{2d,d})} \ar[d]\ar[r]^-{\cong}                 & {\rKO^{2q+1}(S^{2d}) }\ar[d]\\
        {\Kh^{2q+1,q}(U_k)      } \ar[d]\ar@{-->}[r]^-{k_h^{2q+1,q}}    & {\KO^{2q+1}(U_k)     }\ar[d]\\
        {\Kh^{2q+1,q}(U_{k+1})  } \ar[d]\ar[r]^-{\cong}                 & {\KO^{2q+1}(U_{k+1}) }\ar[d]\\
        {\rKh^{2q+2,q}(S^{2d,d})} \ar[d]\ar@{>->}[r]                    & {\rKO^{2q+2}(S^{2d}) }\ar[d] \\
        {\Kh^{2q+2,q}(U_k)      } \ar[d]\ar@{-->}[r]^-{k_h^{2q+2,q}}    & {\KO^{2q+2}(U_k)     }\ar[d]\\
        {\Kh^{2q+2,q}(U_{k+1})  } \ar[d]\ar@{>->}[r]                    & {\KO^{2q+2}(U_{k+1}) }\ar[d]\\
                                        \ar@{}[r]^-{\cdots}             &
        }
\end{equation*}
\caption{The inductive step.}
\label{fig:proof}
\end{figure}

\textbf{The twisted case.}
To obtain the theorem in the case of coefficients in a vector bundle \( \vb{E} \) over \( X \), we replace the exact triangle above by the triangle
\begin{align*}
   \SgmInf\Thom(\lb{E}_{|U_{k+1}})\rightarrow\SgmInf\Thom(\lb{E}_{|U_k})\rightarrow\SgmInf\Thom(\vb{E}_{|C_k}\oplus\N_{C_k\backslash U_k})\rightarrow S^{1,0}\smsh\dots
\end{align*}
The existence of this exact triangle is shown in the next lemma. The Thom space on the right is again just a sphere, so we can proceed as in the untwisted case.

\begin{lem}\label{lem:ThomLocalizationTriangle}
   Given a smooth subvariety \( Z \) of a smooth variety \( X \) with complement \( U \), and given any vector bundle \( \vb{E} \) over \( X \), we have an exact triangle
   \begin{align*}
   \SgmInf\Thom(\vb{E}_{|U})\rightarrow\SgmInf\Thom\vb{E}\rightarrow\SgmInf\Thom(\vb{E}_{|Z}\oplus\N_{Z\backslash X})\rightarrow S^{1,0}\smsh\dots
   \end{align*}
\end{lem}
\begin{proof}
From the Thom isomorphism theorem we know that the Thom space of a vector bundle over a smooth base is \( \A^1 \)-weakly equivalent to the quotient of the vector bundle by the complement of the zero section.
Consider the closed embeddings \( U\hookrightarrow(\vb{E}-Z) \), \( X\hookrightarrow\vb{E} \) and \( Z\hookrightarrow\vb{E} \). Computing the normal bundles, we obtain
\begin{align*}
   (\vb{E}-Z)/(\vb{E}-X)&\cong\Thom_U(\vb{E}_{|U})\\
   \vb{E}/(\vb{E}-X)&\cong\Thom_X{\vb{E}} \\
   \vb{E}/(\vb{E}-Z)&\cong\Thom_Z(\vb{E}_{|Z}\oplus\N_{Z\backslash X})
\end{align*}
The claim follows by passing to the stable homotopy category and applying the octahedral axiom to the composition of the embeddings \( {(\vb{E}-X)\subseteq}{(\vb{E}-Z)\subseteq\vb{E}} \).
\end{proof}

\textbf{Remaining details concerning a point.}
To finish the proof of Theorem~\ref{thm:extendedversion}, we now return to the maps of degrees \( 1 \) and \( -2 \) in the case of a point, which we skipped above. First, let us deal with degree \( 1 \). The odd KO-groups of a point are all trivial except for \( \KO^{-1} \), so \( k_h^{2q-1,q} \) is trivially a surjection unless \( q\equiv 0 \mod 4 \). In that case, surjectivity of  \( k_h^{-1,0} \) is clear from the following diagram:
\begin{align*}
\xymatrix{
        {\dots}\ar[r] & {\Kh^{-1,0}} \ar[r] \ar[d]^{k_h^{-1,0}} & {\GW^{-1}}\ar[r]\ar[d]^{\cong} & {\K_0}\ar[r]\ar[d]^{\cong} & {\dots} \\
        {\dots}\ar[r] & {\KO^{-1}}    \ar[r] \ar@{}[drr]|{=}     & {\KO^{-2}}\ar[r]               & {\K^0}\ar[r]               & {\dots}
\\
        {\dots}\ar[r] & {\Kh^{-1,0}}\ar@{->>}[r] \ar[d]^{k_h^{-1,0}} & {\Z/2}\ar[r]^-{0} \ar[d]^{\cong} & {\Z} \ar[r] \ar[d]^{\cong} & {\dots} \\
        {\dots}\ar[r] & {\Z/2}       \ar[r]^{\cong}                   & {\Z/2}\ar[r]^-{0}                & {\Z} \ar[r]                & {\dots}
}
\end{align*}
Lastly, we consider what happens in degree \( -2 \). Again, three out of four cases are trivial as \( \Kh^{2q+2,q}=\W^{q+2} \) is zero unless \( q\equiv 2 \mod 4 \). For the non-trivial case, consider the map \( \eta \) appearing in triangle~\eqref{eq:KKOtriang}. As the negative algebraic K-groups of \( \C \) are zero, \( \eta \) yields automorphisms of \( \W^{p-q} \) in negative degrees. In topology, the corresponding maps are given by multiplication with a generator \( \eta \) of \( \KO^{-1} \), and \( \eta^2 \) generates \( \KO^{-2} \). So the commutative square
\begin{align*}
\xymatrix{
        {\W^0}\ar[r]^{\cong}\ar[d]^{\cong} & {\W^0} \ar[d]^{k_h^{0,-2}}\\
        {\KO^{-1}}\ar[r]^-{\eta}_-{\cong}        & {\KO^{-2}}
        }
\end{align*}
shows that \( k_h^{0,-2} \) is an injection (in fact, an isomorphism), as claimed. This completes the proof of Theorem~\ref{thm:extendedversion}.
\bigskip
\begin{rem}\label{rem:old-proof}
We indicate briefly how Theorem~\ref{thm:mainthm} can alternatively be obtained by working only with the maps in degrees \( 0 \) and \( -1 \) that can be defined by more elementary means. The basic strategy ---~comparing the localization sequences arising from the inclusion of a closed cell \( C_k \) into the union of ``higher'' cells \( U_{k} \)~--- still works. But we cannot deduce that the comparison maps are isomorphisms on \( U_k \) from the fact that they are isomorphisms on \( U_{k+1} \) because the parts of the sequences that we can actually compare are now too short. We can, however, still deduce that the maps in degree \( 0 \) with domains the Grothendieck-Witt groups of \( U_k \) are surjective, and that the maps in degree \( -1 \) with domains the Witt groups of \( U_k \) are injective. The inductive step can then be completed with the help of the Bott/Karoubi sequences. This argument works even without assuming that the comparison maps are compatible with the boundary maps in localization sequences in general: in the relevant cases the cohomology groups involved are so simple that this property can be checked by hand.
\end{rem}

\section{The Atiyah-Hirzebruch spectral sequence}\label{sec:AHSS}
We now aim to prepare the ground for the discussion of the KO-theory of some examples in the next section.
The main computational tool will be the Atiyah-Hirzebruch spectral sequence, which in topology exists for any generalized cohomology theory and any finite-dimensional CW complex \( X \) \citelist{\cite{Adams}*{III.7}\cite{Kochman}*{Theorem~4.2.7}}.
For KO-theory, it has the form
\begin{equation*}
 E^{p,q}_2=H^p(X;\KO^q(\point))\Rightarrow \KO^{p+q}(X)
\end{equation*}
with differential \( d_r \) of bidegree \( (r,-r+1) \).
The \( E_2 \)-page is thus concentrated in the half-plane \( p\geq 0 \) and \( 8 \)-periodic in \( q \):
we have the integral cohomology of \( X \) in rows \( q\equiv 0 \) and \( q\equiv-4 \) mod~\( 8 \), its cohomology with \( \Z/2 \)-coefficients in rows \( q\equiv -1 \) and \( q\equiv -2 \), and all other rows are zero.
The differential \( d_2 \) is given by \( \Sq^2\circ\pi_2 \) and \( \Sq^2 \) on rows \( q\equiv 0 \) and \( q\equiv-1 \), respectively, where
\begin{equation*}
 \Sq^2\colon{H^*(X;\Z/2)\rightarrow H^{*+2}(X;\Z/2)}
\end{equation*}
is the second Steenrod square and \( \pi_2 \) is mod-\( 2 \) reduction \cite{Fujii:P}*{1.3}.

The spectral sequence is multiplicative \cite{Kochman}*{Proposition~4.2.9}. That is, the multiplication on the \( E_2 \)-page induced by the cup product on singular cohomology and the ring structure of \( \KO^*(\point) \) (see~\eqref{eq:KO-coefficients}) descends to a multiplication on all subsequent pages, such that the multiplication on the \( E_{\infty} \)-page is compatible with the multiplication on \( \KO^*(X) \). In particular, each page is a module over \( \KO^*(\point) \). The differentials of the spectral sequence are derivations, \ie they satisfy a Leibniz rule.

\subsection{The Atiyah-Hirzebruch spectral sequence for cellular varieties}\label{sec:AHSS:cellular}
For cellular varieties, or more generally for CW complexes with only even-dimensional cells, the spectral sequence becomes simple enough to make some general deductions. We summarize some lemmas of Hoggar and Kono and Hara.
\begin{lem}{\cite{Hoggar}*{2.1 and 2.2}}\label{lem:structure_of_KO}
Let \( X \) be a CW complex with only even-dimensional cells. Then:
\begin{itemize}
\item The ranks of the free parts of \( \KO^0\!{X} \) and \( \KO^{4}\!{X} \) are equal to the number \( t_0 \) of cells of \( X \) of dimension a multiple of \( 4 \).
\item The ranks of the free parts of \( \KO^2\!{X} \) and \( \KO^{6}\!{X} \) are equal to the number \( t_1 \) of cells of \( X \) of dimension \( 2 \) modulo \( 4 \).
\item The groups of odd degrees are two-torsion, \ie \( \KO^{2i-1}\!{X} = (\Z/2)^{s_i} \) for some \( s_i \).
\item \( \KO^{2i}\!X \) is isomorphic to the direct sum of its free part and \( \KO^{2i+1}\!{X} \).
\end{itemize}
Table~\ref{table:eg:notation} in Section~\ref{sec:eg:notation} summarizes these statements.
\end{lem}
\begin{proof}
The cohomology of \( X \) is free on generators given by the cells and concentrated in even degrees.
The first two statements thus follow easily from the Atiyah-Hirzebruch spectral sequence for KO-theory (\eg after tensoring with \( \Q \)). On the other hand, we see from the Atiyah-Hirzebruch spectral sequence for complex K-theory that \( \K^0(X) \) is a free abelian group on the cells while \( \K^1(X) \) is zero. The last two statements thus become consequences of the Bott sequence  \eqref{seq:Bott}.
\end{proof}

The free part of \( \KO^* \) is thus very simple.
In good cases, the spectral sequence also provides a nice description of the \( 2 \)-torsion.
To see this, note that \( \Sq^2\Sq^2=\Sq^3\Sq^1 \) must vanish when the cohomology of \( X \) with \( \Z/2 \)-coefficients is concentrated in even degrees. So we can view \( (H^*(X;\Z/2),\Sq^2) \) as a differential graded algebra over \( \Z/2 \). To lighten the notation, we will write
\begin{equation*}
 H^*(X,\Sq^2):=H^*(H^*(X;\Z/2),\Sq^2)
\end{equation*}
for the cohomology of this algebra. We keep the same grading as before, so that it is concentrated in even degrees. The row \( q\equiv -1 \) on the \( E_3 \)-page is given by \( H^*(X,\Sq^2)\cdot\eta \), where \( \eta \) is the generator of \( \KO^{-1}(\point) \). Since it is the only row that contributes to \( \KO^* \) in odd degrees, we arrive at the following lemma, which will be central to our computations.
\begin{lem}\label{lem:2-torsion_of_KO}
 Let \( X \) be as above.
 If the Atiyah-Hirzebruch spectral sequence of \( \KO^*(X) \) degenerates on the \( E_3 \)-page, then
 \begin{equation*}
  \KO^{2i-1}(X)\cong\bigoplus_k H^{2i+8k}(X,\Sq^2)
 \end{equation*}
\end{lem}

In all the examples we consider below, the spectral sequence does indeed degenerate at this stage. However, showing that it does can be tricky. One step in the right direction is the following observation of Kono and Hara \cite{KonoHara:Gr}*{Proposition~1}.

\begin{lem}\label{lem:AHSS_first-differential}
 Let \( X \) be as above. If the differentials \( d_3 \), \( d_4 \), \( \dots \), \( d_{r-1} \) are trivial and \( d_r \) is non-trivial, then \( r\equiv 2\mod 8 \). In other words, the first non-trivial differential after \( d_2 \) can only appear on a page \( E_r \) with page number \( r\equiv 2\mod 8 \).

 Such a differential is non-zero only on rows \( q\equiv 0 \) and \( q\equiv -1 \) mod~\( 8 \). If it is non-zero on some \( x \) in row \( q\equiv 0 \), then it is also non-zero on \( \eta x \) in row \( q\equiv -1 \).
 Conversely, if it is non-zero on some \( y \) in row \( q\equiv -1 \), there exists some \( x \) in row \( q\equiv 0 \) such that \( y=x\eta \) and \( d_r \) is non-zero on \( x \).
\end{lem}
\begin{proof}
We see from the spectral sequence of a point that \( d_r\eta=0 \) for all differentials.
Thus, multiplication by \( \eta \) gives a map of bidegree \( (0,-1) \) on the spectral sequence that commutes with the differentials.
On the \( E_2 \)-page this map is mod-\( 2 \) reduction from row \( q\equiv 0 \) to row \( q\equiv -1 \) and the identity between rows \( q\equiv -1 \) and \( q\equiv -2 \). It follows that on the \( E_3 \)-page multiplication by \( \eta \) induces a surjection from row \( q\equiv 0 \) to row \( q\equiv -1 \) and an injection of row \( q\equiv -1 \) into row \( q\equiv -2 \). This implies all statements above.
\end{proof}

We derive a corollary that we will use to deduce that the spectral sequence collapses for certain Thom spaces:
\begin{cor}\label{cor:Xcollaps-Tcollapse}\label{sec:AHSS:Thom}
 Suppose we have a continuous map \( p\colon{X\rightarrow T} \) of CW complexes with only even-dimensional cells. Suppose further that the Atiyah-Hirzebruch spectral sequence for \( \KO^*(X) \) collapses on the \( E_3 \)-page, and that \( p^* \) induces an injection in row \( q\equiv-1 \):
\begin{equation*}
 p^*\colon{H^*(T,\Sq^2)\hookrightarrow H^*(X,\Sq^2)}
\end{equation*}
 Then the spectral sequence for \( \KO^*(T) \) also collapses at this stage.
\end{cor}
\begin{proof}
Write \( d_r \) for the first non-trivial higher differential, so \( r\equiv2\mod 8 \). Then, for any element \( x \) in row \( q\equiv 0 \), we have \( p^*(d_rx)=d_rp^*(x)=0 \) since the spectral sequence for \( X \) collapses. From our assumption on \( p^* \) we can deduce that \( d_rx=0 \). By the preceding lemma, this is all we need to show.
\end{proof}

\subsection{The Atiyah-Hirzebruch spectral sequence for Thom spaces}
In order to compute twisted KO-groups, we need to apply the Atiyah-Hirzebruch spectral sequence of KO-theory to Thom spaces. So let \( X \) be a finite-dimensional CW complex, and let \( \pi\colon\vb{E}\rightarrow X \) be a vector bundle of constant rank over \( X \).
Though we will be mainly interested in the case when \( \vb{E}\) is complex, we may more generally assume here that \( \vb{E} \) is any real vector bundle which is oriented.
Then the Thom isomorphism for singular cohomology tells us that the reduced cohomology of the Thom space \( \Thom{\vb E} \) is additively isomorphic to the cohomology of \( X \) itself, apart from a shift in degrees by \(r:=\rank_{\R}\vb{E} \). The isomorphism is given by multiplication with a Thom class \(\theta\) in \(\rH^{r}(\Thom{\vb E};\Z)\):
\begin{equation*}
\begin{aligned}
 H^*(X;\Z) &\xrightarrow{\cong} \rH^{*+r}(\Thom\vb E;\Z)\\
     x\quad&\mapsto \quad\pi^*(x)\cdot\theta
\end{aligned}
\end{equation*}
Similarly, the reduction of \(\theta\) modulo two induces an isomorphism of the respective singular cohomology groups with \(\Z/2\)-coefficients. Thus, apart from a shift of columns, the entries on the \( E_2 \)-page of the spectral sequence for \( \rKO^*(\Thom{\vb{E}}) \) are identical to those on the \( E_2 \)-page for \( \KO^*(X) \). However, the differentials may differ.
\begin{lem}\label{lem:Sq-on-Thom}
 Let \( \vb{E}\overset{\pi}{\rightarrow}X \) be a complex vector bundle of constant rank over a topological space \( X \), with Thom class \( \theta \) as above.
 The second Steenrod square on \( \rH^*(\Thom\vb E;\Z/2) \) is given by \mbox{``\,\( {\Sq^2}+ c_1(\vb{E}) \)''}, where \( c_1(\vb{E}) \) is the first Chern class of \( \vb{E} \) modulo two. That is,
 \begin{equation*}
  \Sq^2(\pi^*x\cdot\theta)=\pi^*\left(\Sq^2(x)+c_1(\vb{E})x\right)\cdot\theta
 \end{equation*}
 for any \( x\in H^*(X;\Z/2) \). More generally, if \(\vb{E}\) is a real oriented vector bundle, the second Steenrod square on the cohomology of its Thom space is given by \mbox{``\,\( {\Sq^2}+ w_2(\vb{E}) \)''}, where \( w_2 \) is the second Stiefel-Whitney class of \(\vb{E}\).
\end{lem}
\begin{proof}
This is a special case of an identity of Thom, which he in fact used to \emph{define} Stiefel-Whitney classes:
\begin{equation*}
  \Sq^i(\pi^*x\cdot\theta)=\pi^*\left(\Sq^i(x)+w_i(\vb{E})x\right)\cdot\theta
\end{equation*}
See \cite{MilnorStasheff}*{page~91}.
\end{proof}

When \( X \) is a CW complex with cells only in even dimensions, the operation \( {\Sq^2}+c_1 \) can be viewed as a differential on  \( H^*(X;\Z/2) \) for any \( c_1 \) in \( H^2(X;\Z/2) \). Extending our previous notation, we denote the cohomology with respect to this differential by
\begin{equation}\label{eq:notation:Sq-cohomology}
 H^*(X,{\Sq^2}+c_1):=H^*(H^*(X;\Z/2),{\Sq^2}+c_1)
\end{equation}
\begin{cor}[of Lemmas \ref{lem:2-torsion_of_KO} and \ref{lem:Sq-on-Thom}]\label{cor:2-torsion_of_Thom-KO}
  If the Atiyah-Hirzebruch spectral sequence of \( \rKO^*(\Thom{\vb{E}}) \) degenerates on the \( E_3 \)-page, then
 \begin{equation*}
  \KO^{2i-1}(X;\vb{E})\cong\bigoplus_k H^{2i+8k}(X,{\Sq^2}+c_1\vb{E})
 \end{equation*}
\end{cor}

It is true more generally that the differentials in the spectral sequence for \( \rKO^*(\Thom{\vb{E}}) \) depend only on the second Stiefel-Whitney class of \( \vb{E} \).
This follows from the observation that the Atiyah-Hirzebruch spectral sequence is compatible with Thom isomorphisms, as is made more precise by the next lemma:

Fix a vector bundle \( \vb{E} \) of constant rank \( r \) over a finite-dimensional CW complex \( X \). Suppose \( \vb{E} \) is oriented with respect to ordinary cohomology and let \( \theta\in\rH^*(\Thom{\vb{E}};\Z) \) be a Thom class.
\begin{lem}\label{lem:AHSS-Thom-iso}
If \( \vb{E} \) is oriented with respect to \( \KO^* \), then \( \theta \) survives to the \( E_{\infty} \)-page of the Atiyah-Hirzebruch spectral sequence computing \( \rKO^*(\Thom{\vb{E}}) \), and the Thom isomorphism for \( H^* \) extends to an isomorphism of spectral sequences.
That is, for each page right multiplication with the class of \( \theta \) in \( \rE^{r,0}_s(\Thom{\vb{E}}) \) gives an isomorphism of \( E^{*,*}_s(X) \)-modules
 \begin{equation*}
  E^{*,*}_s(X)\overset{\cdot\theta}{\underset{\cong}\longrightarrow} \rE^{*+r,*}_s(\Thom{\vb{E}})
 \end{equation*}
Moreover, any lift of \( \theta\in \rE^{r,0}_{\infty}(\Thom{\vb{E}}) \) to \( \rKO^r(\Thom{\vb{E}}) \) defines a Thom class of \( \vb{E} \) with respect to \( \KO^* \).
The isomorphism of the \( E_{\infty} \)-pages of the spectral sequences is induced by the Thom isomorphism given by multiplication with any such class.
\end{lem}
\begin{proof}
We may assume without loss of generality that \(X\) is connected. Fix a point \( x \) on \( X \). The inclusion of the fibre over \( x \) into \( \vb{E} \) induces a map \( i_x\colon{S^r\hookrightarrow\Thom{\vb{E}}} \).
By assumption, the pullback \( i^*_x \) on ordinary cohomology maps \( \theta \) to a generator of \( \rH^r(S^r) \), and the pullback on \( \rKO^* \) gives a surjection
\begin{equation*}
 \rKO^*(\Thom{\vb{E}})\overset{i_x^*}{\twoheadrightarrow}\rKO^r(S^r)
\end{equation*}
Consider the pullback along \( i_x \) on the \( E_{\infty} \)-pages of the spectral sequences for \( S^r \) and \( \Thom{\vb{E}} \).
Since we can identify
\( \rE_{\infty}^{r,0}(\Thom{\vb{E}}) \)
with a quotient of
\( \rKO^r(\Thom{\vb{E}}) \)
and
\( \rE_{\infty}^{r,0}(S^r) \)
with
\( \rKO^r(S^r) \),
we must have a surjection
\begin{equation*}
 i_x^*\colon{\rE_{\infty}^{r,0}(\Thom{\vb{E}})\twoheadrightarrow \rE_{\infty}^{r,0}(S^r)}
\end{equation*}
On the other hand, the behaviour of \( i_x^* \) on \( \rE_{\infty}^{r,0} \) is determined by its behaviour on \( \rH^r \), whence we can only have such a surjection if \( \theta \) survives to the \( \rE_{\infty} \)-page of \( \Thom{\vb{E}} \).
Thus, all differentials vanish on \( \theta \), and if multiplication by \( \theta \) induces an isomorphism from \( E_s^{*,*}(X) \) to \( \rE_s^{*+r,*} \) on page \( s \), it also induces an isomorphism on the next page.
Lastly, consider any lift of \( \theta \) to an element \( \Theta \) of \( \rKO^r(\Thom{\vb{E}}) \).
It is clear by construction that right multiplication with \( \Theta \) gives an isomorphism from \( E_{\infty}(X) \) to \( \rE_{\infty}(\Thom{E}) \), and thus it also gives an isomorphism from \( \KO^*(X) \) to \( \rKO^*(\Thom{\vb{E}}) \).
Thus, \( \Theta \) is a Thom class for \( \vb{E} \) with respect to \( \KO^* \).
\end{proof}

Lemma~\ref{lem:AHSS-Thom-iso} allows the following strengthening of Lemma~\ref{lem:KOtwists}:
\begin{cor}\label{cor:Identifaction_of_Thom-AHSSs}
For complex vector bundles \( \vb{E} \) and \( \vb{F} \) over \( X \) with identical first Chern class modulo \( 2 \), the spectral sequences computing \( \rKO^*(\Thom{\vb{E}}) \) and \( \rKO^*(\Thom{\vb{F}}) \) can be identified up to a possible shift of columns when \( \vb{E} \) and \( \vb{F} \) have different ranks.
\end{cor}

\pagebreak
\section{Examples}\label{sec:examples}
\newcommand{\varGr}[2]{{\Gr_{#1,#2}}}
\newcommand{\GrSp}[2]{{X_{#1}}} 
\newcommand{\GrSOc}[2]{{S_{#1}}} 
\newcommand{\EIII}{\mathrm{EIII}}
\newcommand{\EVII}{\mathrm{EVII}}
We now turn to the study of projective homogeneous varieties, that is, varieties of the form \( G/P \) for some complex simple linear algebraic group \( G \) with a parabolic subgroup \( P \). Any such variety has a cell decomposition \cite{BGG:Schubert-cells}*{Proposition~5.1}, so that our comparison theorem applies.
As far as we are only interested in the topology of \( G/P \), we may alternatively view it as a homogeneous space for the compact real Lie group \( G^c \) corresponding to \( G \):

\begin{prop}
 Let \( P \) be a parabolic subgroup of a simple complex algebraic group \( G \). Then we have a diffeomorphism
\begin{equation*}
\quotient{G}{P} \cong \quotient{G^c}{K}
\end{equation*}
where \( K \) is a compact subgroup of maximal rank in a maximal compact subgroup \( G^c \) of \( G \). More precisely, \( K \) is a maximal compact subgroup of a Levi subgroup of \( P \).
\end{prop}
\begin{proof}
The Iwasawa decomposition for \( G \) viewed as a real Lie group implies that we have a diffeomorphism \( G\cong G^c\cdot P \) \cite{OnishchikVinberg}*{Ch.~6, Prop.~1.7}, inducing a diffeomorphism of quotients as claimed for \( K=G^c\cap P \).
Since \( G^c\hookrightarrow G \) is a homotopy equivalence, so is the inclusion \( G^c\cap P\hookrightarrow P \).
On the other hand, if \( L \) is a Levi subgroup of \( P \) then \( P=U \rtimes L \), where \( U \) is unipotent and hence contractible. So the inclusion \( L\hookrightarrow P \) is also a homotopy equivalence. It follows that any maximal compact subgroup \( L^c \) of \( L \) is also maximal compact in \( P \), and conversely that any maximal compact subgroup of \( P \) will be contained as a maximal compact subgroup in some Levi subgroup of \( P \). We may therefore assume that \( K \subset L^c \subset L \subset P \) and conclude that \( K\hookrightarrow L^c \) is a homotopy equivalence. Since both groups are compact, it follows that in fact \( K\cong L^c \).
\end{proof}

The KO-theory of homogeneous varieties has been studied intensively. In particular, the papers \cite{KonoHara:Gr} and \cite{KonoHara:HSS} of Kono and Hara provide complete computations of the (untwisted) KO-theory of all compact irreducible hermitian symmetric spaces, which we list in Table~\ref{table:HSS}. For the convenience of the reader, we indicate how each of these arises as a quotient of a simple complex algebraic group \( G \) by a parabolic subgroup \( P \), describing the latter in terms of marked nodes on the Dynkin diagram of \( G \) as in \cite{FultonHarris}*{\S~23.3}. The last column gives an alternative description of each space as a quotient of a compact real Lie group.

\begin{table}[b!t!]
\begin{center}
\newlength{\textcolwidth}
\settowidth{\textcolwidth}{Maximal symplectic G}
\newcommand{\textcol}[1]{\parbox{\textcolwidth}{\vspace{4pt}\raggedright #1\vspace{4pt}}}
\newcommand{\Pdiag}[1]{
   \ensuremath{\xymatrix@R-=0pt@C-=4pt{#1}}
}
\newcommand{\nodeP}{\circ}  
\newcommand{\nodeG}{\bullet}
\newcommand{\PdiagAIII}{\Pdiag{
*{\nodeP}\ar@{-}[r] & {\cdots} \ar@{-}[r] & *{\nodeP} \ar@{-}[r] & *{\nodeG}\ar@{-}[r] & *{\nodeP} \ar@{-}[r] & \cdots \ar@{-}[r] & *{\nodeP}\\
*{\text{\scriptsize \( 1 \)}} & & & *{\text{\scriptsize \( n \)}} & & & \save *{\text{\scriptsize \( n \)+\( m \)--\( 1 \)}} \restore
}}
\newcommand{\PdiagCI}{\Pdiag{
*{\nodeP}\ar@{-}[r] & *{\nodeP}\ar@{-}[r] & {\cdots} \ar@{-}[r] & *{\nodeP} \ar@{-}[r] & *{\nodeP}\ar@{=}[rr]|-{<} & & *{\nodeG}
}}
\newcommand{\PdiagBDI}{
\Pdiag{
*{\nodeG}\ar@{-}[r] & *{\nodeP} \ar@{-}[r] & {\cdots} \ar@{-}[r] & *{\nodeP}\ar@{=}[rr]|-{>} & & *{\nodeP} & *{\text{\scriptsize (\( n \) odd)}}\\
\\
 & & & & *{\nodeP} \\
*{\nodeG}\ar@{-}[r] & *{\nodeP} \ar@{-}[r] & {\cdots} \ar@{-}[r] & *{\nodeP}\ar@{-}[ru]\ar@{-}[rd] &  & &  *{\text{\scriptsize (\( n \) even)}}\\
 & & & & *{\nodeP}
}}
\newcommand{\PdiagDIII}{\Pdiag{
 & & & & *{\nodeP} \\
*{\nodeP}\ar@{-}[r] & *{\nodeP} \ar@{-}[r] & {\cdots} \ar@{-}[r] & *{\nodeP}\ar@{-}[ru]\ar@{-}[rd] &  & \\
 & & & & *{\nodeG}
}}
\newcommand{\PdiagEIII}{\Pdiag{
  & & *{\nodeP}\ar@{-}[dd] & & & & & \\ \\
  *{\nodeP}\ar@{-}[r] & *{\nodeP}\ar@{-}[r] & *{\nodeP}\ar@{-}[r] & *{\nodeP}\ar@{-}[r] & *{\nodeG} \\
}}
\newcommand{\PdiagEVII}{\Pdiag{
  & & *{\nodeP}\ar@{-}[dd] & & & & &\\ \\
  *{\nodeP}\ar@{-}[r] & *{\nodeP}\ar@{-}[r] & *{\nodeP}\ar@{-}[r] & *{\nodeP}\ar@{-}[r] & *{\nodeP}\ar@{-}[r] & *{\nodeG}\\
}}
  \renewcommand{\arraystretch}{2}
\begin{adjustwidth}{-2cm}{-2cm}
\begin{center}
\begin{tabular}[t]{lM|Ml|M}
  \toprule
   & \quotient{G}{P} & G & Diagram of \( P \) & \quotient{G^c}{K} \\
  \midrule
  \textcol{Grassmannians (AIII)}
  & \varGr{m}{n}
  & \SL_{m+n} & \PdiagAIII
  & \dfrac{\mathrm{U}(m+n)}{\mathrm{U}(m)\times\mathrm{U}(n)}\\
  \textcol{Maximal symplectic Grassmannians (CI)} & \GrSp{n}{2n}
  & \Sp_{2n} & \PdiagCI
  & \quotient{\Sp(n)}{\mathrm{U}(n)}\\
  \textcol{Projective quadrics of dimension \( n\geq 3 \) (BDI)} & Q^n
  & \SO_{n+2} & \PdiagBDI
  & \dfrac{\SO(n+2)}{\SO(n)\times\SO(2)}\\
  \textcol{Spinor varieties (DIII)} & \GrSOc{n}{2n}
  & \SO_{2n} & \hspace{-3pt}\raisebox{10pt}{\PdiagDIII}
  & \quotient{\SO(2n)}{\U(n)}\\
  \textcol{Exceptional hermitian symmetric spaces:} & \EIII
  & E_6 & \PdiagEIII
  & \dfrac{E_6^c}{\Spin(10)\cdot S^1}\\[-12pt]
  & & & & {\scriptstyle(\Spin(10)\cap S^1 = \Z/4)}\\[0pt]
  \textcol{} & \EVII
  & E_7 & \PdiagEVII
  & \dfrac{E_7^c}{E_6^c\cdot S^1}\\[-12pt]
  & & & &{\scriptstyle(E_6^c\cap S^1 = \Z/3)}\\[0pt]
  \bottomrule
\end{tabular}
\end{center}
\end{adjustwidth}
\caption{List of irreducible compact hermitian symmetric spaces. The symbols AIII, CI, \dots{} refer to E.~Cartan's classification. In the description of \( \varGr{m}{n} \) we use \( \mathrm{U}(m+n) \) instead of \( G^c=\SU(m+n) \).
}
\label{table:HSS}
\end{center}
\end{table}

On the following pages, we will run through this list of examples and, in each case, extend Kono and Hara's computations to include KO-groups twisted by a line bundle. Since each of these spaces is a ``Grassmannian'' in the sense that the parabolic subgroup \( P \) in \( G \) is maximal, its Picard group is free abelian on a single generator. Thus, there is exactly one non-trivial twist that we need to consider. In most cases, we ---~ reassuringly~--- recover results for Witt groups that are already known. In a few other cases, we consider our results new.

\pagebreak[3]
The untwisted KO-theory of complete flag varieties is also known in all three classical cases thanks to Kishimoto, Kono and Ohsita. We do not reproduce their result here but instead refer the reader directly to \cite{KKO:Flags}. By a recent result of Calm{\`e}s and Fasel, all Witt groups with non-trivial twists vanish for these varieties \cite{CalmesFasel:Trivial}.

\subsection{Notation}\label{sec:eg:notation}
Topologically, a cellular variety is a CW complex with cells only in even (real) dimensions. For such a CW complex \( X \) the KO-groups can be written in the form displayed in Table~\ref{table:eg:notation} below.
This was shown in Section~\ref{sec:AHSS:cellular} in the case when the twist \( \lb{L} \) is trivial, and the general case follows:
if \( X \) is a CW complex with only even-dimensional cells, so is the Thom space of any complex vector bundle over \( X \)
\cite{MilnorStasheff}*{Lemma~18.1}.

In the following examples, results on \( \KO^{\ast} \) will be displayed by listing the values of the \( t_i \) and \( s_i \). Since the \( t_i \) are just given by counting cells, and since the numbers of odd- and even-dimensional cells of a Thom space \( \Thom_X\!\vb{E} \) only depend on \( X \) and the rank of \( \vb{E} \), the \( t_i \) are in fact independent of \( \lb{L} \). The \( s_i \), on the other hand, certainly will depend on the twist, and we will sometimes acknowledge this by writing \( s_i(\lb{L}) \).

\begin{table}[tbhp]
\renewcommand{\arraystretch}{1.3}
\begin{center}
\begin{tabular}[t]{>{\( }l<{ \)}@{  =  }>{\( }r<{ \)}@{  =  }>{\( }r<{ \)}}
 \toprule
 \KO^6(X;\lb{L}) & \Z^{t_1}\oplus (\Z/2)^{s_0} & \GW^3(X;\lb{L}) \\
 \KO^7(X;\lb{L}) &                (\Z/2)^{s_0} &  \W^0(X;\lb{L}) \\
 \KO^0(X;\lb{L}) & \Z^{t_0}\oplus (\Z/2)^{s_1} & \GW^0(X;\lb{L}) \\
 \KO^1(X;\lb{L}) &                (\Z/2)^{s_1} &  \W^1(X;\lb{L}) \\
 \KO^2(X;\lb{L}) & \Z^{t_1}\oplus (\Z/2)^{s_2} & \GW^1(X;\lb{L}) \\
 \KO^3(X;\lb{L}) &                (\Z/2)^{s_2} &  \W^2(X;\lb{L}) \\
 \KO^4(X;\lb{L}) & \Z^{t_0}\oplus (\Z/2)^{s_3} & \GW^2(X;\lb{L}) \\
 \KO^5(X;\lb{L}) &                (\Z/2)^{s_3} &  \W^3(X;\lb{L}) \\
 \bottomrule
\end{tabular}
\caption{Notational conventions in the examples. Only the \( s_i \) depend on \( \lb{L} \).}
\label{table:eg:notation}
\end{center}
\end{table}

\subsection{Projective spaces}\label{subsec:P}
Complex projective spaces are perhaps the simplest examples for which Theorem~\ref{thm:mainthm} asserts something non-trivial, so we describe the results here separately before turning to complex Grassmannians in general. The computations of the Witt groups of projective spaces were certainly landmark events in the history of the theory. In 1980, Arason was able to show that the Witt group \( \W^0(\P^n) \) of \( \P^n \) over a field \( k \) agrees with the Witt group of \( k \) \cite{Arason}. The shifted Witt groups of projective spaces, and more generally of arbitrary projective bundles, were first computed by Walter in \cite{Walter:PB}. Quite recently, Nenashev deduced the same results via different methods \cite{Nenashev:Q}.

In the topological world, complete computations of \( \KO^i(\CP^n) \) were first published in a 1967 paper by Fujii \cite{Fujii:P}. It is not difficult to deduce the values of the twisted groups \( \KO^i(\CP^n;\OO(1)) \) from these: the Thom space \( \Thom(\OO_{\CP^n}(1)) \) can be identified with \( \CP^{n+1} \), so
\begin{align*}
        \KO^i(\CP^n;\OO(1))      &= \rKO^{i+2}(\Thom(\OO(1)))\\
                                                                &= \rKO^{i+2}(\CP^{n+1})
\end{align*}
Alternatively, we could do all required computations directly following the methods outlined in Section~\ref{sec:AHSS}. The result, in any case, is displayed in Table~\ref{table:result:P}, coinciding with the known results for the (Grothendieck-)Witt groups.

\begin{table}[tbhp]
\begin{center}
\begin{tabular}{>{\( }l<{ \)}|MM|MMMM|MMMM}
\toprule
\KO^*(\CP^n;\lb{L}) &         &         & \multicolumn{4}{M|}{\vb{L}\equiv\OO}  & \multicolumn{4}{M}{\vb{L}\equiv\OO(1)} \\
                   & t_0     & t_1     & s_0 & s_1 & s_2 & s_3       & s_0 & s_1 & s_2 & s_3 \\
 \midrule
 n\equiv 0 \mod 4  & (n/2)+1 & n/2     & 1 & 0 & 0 & 0                        & 1 & 0 & 0 & 0 \\
 n\equiv1          & (n+1)/2 & (n+1)/2 & 1 & 1 & 0 & 0                        & 0 & 0 & 0 & 0 \\
 n\equiv2          & (n/2)+1 & n/2     & 1 & 0 & 0 & 0                        & 0 & 0 & 1 & 0 \\
 n\equiv3          & (n+1)/2 & (n+1)/2 & 1 & 0 & 0 & 1                        & 0 & 0 & 0 & 0 \\
\bottomrule
\end{tabular}
\caption{KO-groups of projective spaces}
\label{table:result:P}
\end{center}
\end{table}

\subsection{Grassmannians}
We now consider the Grassmannians \( \varGr{m}{n} \) of complex \( m \)-planes in \( \C^{m+n} \). Again both the Witt groups and the untwisted KO-groups are already known: the latter by Kono and Hara \cite{KonoHara:Gr}, the former by the work of Balmer and Calm{\`e}s \cite{BalmerCalmes:Gr}. A detailed comparison of the two sets of results in the untwisted case has been carried out by Yagita \cite{Yagita:Gr}. We provide here an alternative topological computation of the twisted groups.

Balmer and Calm{\`e}s state their result by describing an additive basis of the total Witt group of \( \varGr{m}{n} \) in terms of certain ``even Young diagrams''. This is probably the most elegant approach, but needs some space to explain. We will stick instead to the tabular exposition used in the other examples. Let \( \OO(1) \) be a generator of \( \Pic(\varGr{m}{n}) \), say the dual of the determinant line bundle of the universal \( m \)-bundle over \( \varGr{m}{n} \). The result is displayed in Table~\ref{table:results:Gr}.

\begin{table}[b!t!h!]
\begin{center}
\newlength{\firstcolwidth}
\settowidth{\firstcolwidth}{or: \( m\equiv 2 \) and \( n \) odd}
\newcommand{\firstcol}[1]{\parbox{\firstcolwidth}{\vspace{4pt}\raggedright #1\vspace{4pt}}}
\newlength{\testheight}
\settoheight{\testheight}{\( \dfrac{b}{b} \)}

\begin{tabular}{l|MM|MMMM|MMMM}
\toprule
\( \KO^*(\varGr{m}{n};\lb{L}) \) & & & \multicolumn{4}{M|}{\vb{L}\equiv\OO} & \multicolumn{4}{M}{\vb{L}\equiv\OO(1)}\\
                    & t_0               & t_1            & s_0 & s_1 & s_2 & s_3   & s_0 & s_1 & s_2 & s_3 \\
 \midrule
 \firstcol{\( m \) and \( n \) odd s.t. \\ \( \quad m\equiv n \)}
                    & \dfrac{a}{2}      & \dfrac{a}{2}   & b & b & 0 & 0           & 0 & 0 & 0 & 0 \\
 \firstcol{\( m \) and \( n \) odd s.t. \\ \( \quad m\not\equiv n \)}
                    & \dfrac{a}{2}      & \dfrac{a}{2}   & b & 0 & 0 & b           & 0 & 0 & 0 & 0 \\
 \hline
 \firstcol{\(\!\!\!\! \begin{cases}m\equiv n\equiv 0 \\
                         m\equiv 0 \text{ and } n \text{ odd} \\
                         n\equiv 0 \text{ and } m \text{ odd}
          \end{cases} \)}
                    & \dfrac{a+b}{2}    & \dfrac{a-b}{2} & b & 0 & 0 & 0           & b & 0 & 0 & 0 \\
 \firstcol{\(\!\!\!\! \begin{cases} m\equiv n \equiv 2\\
                          m\equiv 2 \text{ and } n \text{ odd} \\
                          n\equiv 2 \text{ and } m \text{ odd}
            \end{cases} \)}
                    & \dfrac{a+b}{2}    & \dfrac{a-b}{2} & b & 0 & 0 & 0           & 0 & 0 & b & 0 \\
 \hline
 \firstcol{\( m\equiv 0 \) and \( n\equiv 2 \)\minrowheight{\dummyfrac}}
                    & \dfrac{a+b}{2}    & \dfrac{a-b}{2} & b & 0 & 0 & 0                        & b_1 & 0 & b_2 & 0 \\
 \firstcol{\( m\equiv 2 \) and \( n\equiv 0 \)\minrowheight{\dummyfrac}}
                    & \dfrac{a+b}{2}    & \dfrac{a-b}{2} & b & 0 & 0 & 0                        & b_2 & 0 & b_1 & 0 \\
 \midrule
 \multicolumn{11}{c}{\parbox{3.7\firstcolwidth}{\(   \)\\
 All equivalences (\( \equiv \)) are modulo \( 4 \). For the values of \( a \) and \( b=b_1+b_2 \), put \( k:=\lfloor\nicefrac{m}{2}\rfloor \) and \( l:=\lfloor\nicefrac{n}{2}\rfloor \). Then
  \begin{align*}
    a   &:= \mm{m+n\\\\m}   & b   &:= \mm{k+l\\\\k}   &
    b_1 &:= \mm{k+l-1\\\\k} & b_2 &:= \mm{k+l-1\\\\k-1}
  \end{align*}
 }}\\
 \bottomrule
\end{tabular}
\caption{KO-groups of Grassmannians}
\label{table:results:Gr}
\end{center}
\end{table}

Our computation is based on the following geometric observation. Let \( \vb{U}_{m,n} \) and \( \vb{U}_{m,n}^{\perp} \) be the universal \( m \)-bundle and the orthogonal \( n \)-bundle on \( \varGr{m}{n} \), so that \( \vb{U}\oplus\vb{U}^{\perp}=\OO^{\oplus (m+n)} \). We have various natural inclusions between the Grassmannians of different dimensions, of which we fix two:
\begin{description}
\item[\( \varGr{m}{n-1}\hookrightarrow\varGr{m}{n} \)]
via the inclusion of the first \mbox{\( m+n-1 \)} coordinates into \( \C^{m+n} \)
\item[\( \varGr{m-1}{n}\hookrightarrow\varGr{m}{n} \)]
by sending an \( (m-1) \)-plane \( \Lambda \) to the \( m \)-plane \( \Lambda\oplus\left<e_{m+n}\right> \), where \( e_1, e_2, \ldots, e_{m+n} \) are the canonical basis vectors of \( \C^{m+n} \)
\end{description}
\begin{lem}\label{lem:Gr-Geometry}
The normal bundle of \( \varGr{m}{n-1} \) in \( \varGr{m}{n} \) is the dual \( \vb{U}^{\dual}_{m,n-1} \) of the universal \( m \)-bundle. Similarly, the normal bundle of \( \varGr{m-1}{n} \) in \( \varGr{m}{n} \) is given by \( \vb{U}^{\perp}_{m-1,n} \).
In both cases, the embeddings of the subspaces extend to embeddings of their normal bundles, such that one subspace is the closed complement of the normal bundle of the other.
\end{lem}

This gives us two cofibration sequences of pointed spaces:
\begin{align}
\label{seq:Gr-cofib1}
 \varGr{m-1}{n}_+ \overset{i}{\hookrightarrow} &\;\varGr{m}{n}_+ \overset{p}{\twoheadrightarrow} \Thom(\vb{U}^{\dual}_{m,n-1}) \\
\label{seq:Gr-cofib2}
 \varGr{m}{n-1}_+  \overset{i}{\hookrightarrow} &\;\varGr{m}{n}_+ \overset{p}{\twoheadrightarrow} \Thom(\vb{U}^{\perp}_{m-1,n})
\end{align}
These sequences are the key to relating the untwisted KO-groups to the twisted ones.
Following the notation in \cite{KonoHara:Gr}, we write \( A_{m,n} \) for the cohomology of \( \varGr{m}{n} \) with \( \Z/2 \)-coefficients, denoting by \( a_i \) and \( b_i \) the Chern classes of \( \vb{U} \) and \( \vb{U}^{\perp} \), respectively, and by \( a \) and \( b \) the total Chern classes \( 1+a_1+\cdots+ a_m \) and \( 1+b_1+\cdots +b_n \):
\begin{equation*}
 A_{m,n}=\cfrac[l]{\Z/2\left[a_1,a_2,\ldots,a_m,b_1,b_2,\ldots b_n\right]}{a\cdot b = 1}
\end{equation*}
We write \( d \) for the differential given by the second Steenrod square \( \Sq^2 \), and \( d' \) for \( {\Sq^2}+a_1 \).
To describe the cohomology of \( A_{m,n} \) with respect to these differentials, it is convenient to introduce the algebra
\begin{equation*}
 B_{k,l}=\cfrac[l]{\Z/2\left[a_2^2,a_4^2,\ldots,a_{2k}^2,b_2^2,b_4^2,\ldots,b_{2l}^2\right]}{(1+a_2^2+\cdots+a_{2k}^2)(1+b_2^2+\cdots+b_{2l}^2)=1}
\end{equation*}
Note that this subquotient of \( A_{2k,2l} \) is isomorphic to \( A_{k,l} \) up to a ``dilatation'' in grading. Proposition~2 in \cite{KonoHara:Gr} tells us that
\begin{equation*}
 H^*(A_{m,n},d)=         \begin{cases}
                          B_{k,l}                               & \text{if $(m,n) = $ \parbox[t]{3.5cm}{$(2k,2l)$, $(2k+1,2l)$ or $(2k,2l+1)$}}\\
                          B_{k,l}\oplus B_{k,l}\cdot a_mb_{n-1} & \text{if $(m,n) = (2k+1,2l+1)$}
                         \end{cases}
\end{equation*}
Here, the algebra structure in the case where both \( m \) and \( n \) are odd is determined by \( (a_mb_{n-1})^2=0 \).

\begin{lem}\label{lem:Sq-cohomology-of-Amn} The cohomology of \( A_{m,n} \) with respect to the twisted differential \( d' \) is as follows:
 \begin{equation*}
   H^*(A_{m,n},d')=\begin{cases}
                            B_{k,l-1}\cdot a_m \oplus B_{k-1,l}\cdot b_n  & \text{ if \( (m,n) = (2k,2l) \)}   \\
                            B_{k,l} \cdot a_m                             & \text{ if \( (m,n) = (2k,2l+1) \)} \\
                            B_{k,l} \cdot b_n                             & \text{ if \( (m,n) = (2k+1,2l) \)} \\
                            0                                             & \text{ if \( (m,n) = (2k+1,2l+1) \)}
                           \end{cases}
 \end{equation*}
\end{lem}
\begin{proof}
Let us shift the dimensions in the cofibration sequences \eqref{seq:Gr-cofib1} and \eqref{seq:Gr-cofib2} in such a way that we have the Thom spaces of \( \vb{U}^{\dual}_{m,n} \) and \( \vb{U}^{\perp}_{m,n} \) on the right. Since the cohomologies of the spaces involved are concentrated in even degrees, the associated long exact sequence of cohomology groups falls apart into short exact sequences.  Reassembling these, we obtain two short exact sequences of differential \( (A_{m,n+1},d) \)- and \( (A_{m+1,n},d) \)-modules, respectively:
\begin{align}
\label{eq:GrA1}
 0\rightarrow (A_{m,n},d')\cdot\theta^{\dual} \overset{p^*}{\longrightarrow} & (A_{m,n+1},d) \overset{i^*}{\longrightarrow} (A_{m-1,n+1},d) \rightarrow 0 \\
\label{eq:GrA2}
 0\rightarrow (A_{m,n},d')\cdot\theta^{\perp} \overset{p^*}{\longrightarrow} & (A_{m+1,n},d) \overset{i^*}{\longrightarrow} (A_{m+1,n-1},d) \rightarrow 0
\end{align}
Here, \( \theta^{\dual} \) and \( \theta^{\perp} \) are the respective Thom classes of \( \vb{U}^{\dual}_{m,n} \) and \( \vb{U}^{\perp}_{m,n} \).
The map \( i^* \) in the first row is the obvious quotient map annihilating \( a_{m} \). Its kernel, the image of \( A_{m,n} \) under multiplication by \( a_m \), is generated as an \( A_{m,n+1} \)-module by its unique element in degree \( 2m \), and thus we must have \( p^*(\theta^{\dual})=a_m \). Likewise, in the second row we have \( p^*(\theta^{\perp})=b_n \).

The lemma can be deduced from here case by case.
For example, when both \( m \) and \( n \) are even, \( i^* \) maps \( H^*(A_{m,n+1},d)=B_{k,l} \) to the first summand of \( H^*(A_{m-1,n+1},d)=B_{k-1,l}\oplus {B_{k-1,l}\cdot a_{m-1}b_n} \) by annihilating \( a_{m}^2 \). We know by comparison with the short exact sequences for the \( A_{m,n} \) that the kernel of this map is \( B_{k,l-1} \) mapping to \( B_{k,l} \) under multiplication by \( a_{m}^2 \). Thus, we obtain a short exact sequence
\begin{equation}\label{seq:aux:Gr}
 0\rightarrow B_{k-1,l}\cdot a_{m-1}b_n \overset{\partial}{\longrightarrow}H^*(A_{m,n},d')\cdot\theta^{\dual}\overset{p^*}{\longrightarrow}B_{k,l-1}\cdot a_{m}^2\rightarrow 0
\end{equation}
For the Steenrod square \(\Sq^2\) of the top Chern class \( a_m \) of \( \vb U\), we have \(\Sq^2(a_m)=a_1a_m\). This can be checked, for example, by expressing \( a_m\) as the product of the Chern roots of \(\vb U\). Consequently, \( d'(a_m)=0 \). Together with the fact that \( H^*(A_{m,n},d') \) is a module over \( H^*(A_{m,n+1},d) \), this shows that we can define a splitting of \( p^* \) by sending \( a_m^2 \) to \( a_m\theta^{\dual} \).  Thus, \( H^*(A_{m,n},d') \) contains \( {B_{k,l-1}\cdot a_m} \) as a direct summand. If instead of working with sequence~\eqref{eq:GrA1} we work with sequence~\eqref{eq:GrA2}, we see that \( H^*(A_{m,n},d') \) also contains a direct summand \( {B_{k-1,l}\cdot b_n} \). These two summands intersect trivially, and a dimension count shows that together they encompass all of \( H^*(A_{m,n},d') \). Alternatively, one may check explicitly that the boundary map \( \partial \) above sends \( a_{m-1}b_n \) to \( b_n\theta \).
The other cases are simpler.
\end{proof}

\begin{lem}
 The Atiyah-Hirzebruch spectral sequence for \mbox{\( \rKO^*(\Thom{\vb{U}^{\dual}_{m,n}}) \)} collapses at the \( E_3 \)-page.
\end{lem}
\begin{proof}
By Proposition~4 of \cite{KonoHara:Gr} we know that the spectral sequence for \( \KO^*(\varGr{m}{n}) \) collapses as this stage, for any \( m \) and \( n \).
Now, if both \( m \) and \( n \) are even, we have
\begin{equation*}
 (B_{k,l-1}\cdot a_m \oplus B_{k-1,l}\cdot b_n)\cdot\theta
\end{equation*}
in the \( (-1)^{\text{st}} \) row of the \( E_3 \)-pages of the spectral sequences for \( \Thom{\vb{U}^{\dual}} \) and \( \Thom{\vb{U}^{\perp}} \),
where \( \theta=\theta^{\dual} \) or \( \theta^{\perp} \), respectively.
In the case of \( \vb{U}^{\dual} \) we see from \eqref{seq:aux:Gr} that \( p^* \) maps the second summand injectively to the \( E_3 \)-page of the spectral sequence for \( \KO^*(\varGr{m}{n+1}) \). Similarly, in the case of \( \vb{U}^{\perp} \), the first summand is mapped injectively to the \( E_3 \)-page of \( \KO^*(\varGr{m+1}{n}) \).
Since the spectral sequences for \( \Thom{\vb{U}^{\dual}} \) and \( \Thom{\vb{U}^{\perp}} \) can be identified via Corollary~\ref{cor:Identifaction_of_Thom-AHSSs}, we can argue as in Corollary~\ref{cor:Xcollaps-Tcollapse} to see that they must collapse at this stage. Again, the cases when at least one of \( m \), \( n \) is odd are similar but simpler.
\end{proof}
We may now apply Corollary~\ref{cor:2-torsion_of_Thom-KO}. The entries of Table~\ref{table:results:Gr} that do not appear in \cite{KonoHara:Gr}, \ie those of the last four columns, follow from Lemma~\ref{lem:Sq-cohomology-of-Amn} by noting that \( B_{k,l} \) is concentrated in degrees \( 8i \) and of dimension \( \dim B_{k,l}=\dim A_{k,l} = \mm{k+l\\k} \).

\subsection{Maximal symplectic Grassmannians}\label{sec:eg:symplecticGr}
The Grassmannian of isotropic \( n \)-planes in \( \C^{2n} \) with respect to a non-degenerate skew-symmetric bilinear form is given by \( \GrSp{n}{2n}=\Sp(n)/U(n) \). The universal bundle \( \vb{U} \) on the usual Grassmannian \( \Gr(n,2n) \) restricts to the universal bundle on \( \GrSp{n}{2n} \), and so does the orthogonal complement bundle \( \vb{U}^{\perp} \). We will continue to denote these restrictions by the same letters. Thus, \( \vb{U}\oplus\vb{U}^{\perp}\cong\C^{2n} \) on \( \GrSp{n}{2n} \), and the fibres of \( \vb{U} \) are orthogonal to those of \( \vb{U}^{\perp} \) with respect to the standard hermitian metric on \( \C^{2n} \). The determinant line bundles of \( \vb{U} \) and \( \vb{U}^{\perp} \) give dual generators \( \OO(1) \) and \( \OO(-1) \) of the Picard group of \( \GrSp{n}{2n} \).

\begin{thm}\label{thm:GrSp-KO}
 The additive structure of \( \KO^*(\GrSp{n}{2n};\lb{L}) \) is as follows:\vspace{-0.5\baselineskip}
 \begin{center}
 \begin{tabular}[t]{>{\( }l<{ \)}|MM|M|M}
 \toprule
                    & t_0     & t_1     & s_i(\OO)                   & s_i(\OO(1))                \\
 \midrule
   \text{\( n \) even}  & 2^{n-1} & 2^{n-1} & \rho(\tfrac{n}{2}, {i})   & \rho(\tfrac{n}{2}, i-n) \\
   \minrowheight{\( \tfrac{n}{n} \)}
   \text{\( n \) odd}   & 2^{n-1} & 2^{n-1} & \rho(\tfrac{n+1}{2}, {i}) & 0                         \\
 \bottomrule
 \end{tabular}
 \end{center}

\vspace{0.5\baselineskip}
\noindent
 Here, for any \( i\in\Z/4 \) we write \( \rho(n,i) \) for the dimension of the \( i \)-graded piece of a \( \Z/4 \)-graded exterior algebra \( \Lambda_{\Z/2}(g_1, g_2, \ldots, g_n) \) on \( n \) homogeneous generators \( g_1 \), \( g_2 \), \dots, \( g_n \) of degree \( 1 \), \ie
 \begin{equation*}
  \rho(n,i)=\sum_{\begin{smallmatrix}{d\equiv i}\\{\mod 4}\end{smallmatrix}}{\genfrac{(}{)}{0pt}{}{n}{d}}
 \end{equation*}
 A table of the values of \( \rho(n, i) \) can be found in \cite{KonoHara:HSS}*{Proposition~4.1}.
\end{thm}

It turns out to be convenient to work with the vector bundle \mbox{\( \vb{U}^{\perp}\oplus\OO \)} for the computation of the twisted groups \( \KO^*(\GrSp{n}{2n}; \OO(1)) \). Namely, we have the following analogue of Lemma~\ref{lem:Gr-Geometry}.

\begin{lem}\label{lem:GrSp-Geometry}
There is an open embedding of the bundle \mbox{ \( \vb{U}^{\perp}\oplus\OO \)} over the symplectic Grassmannian \( \GrSp{n}{2n} \) into the symplectic Grassmannian \( \GrSp{n+1}{2n+2} \) whose closed complement is again isomorphic to \( \GrSp{n}{2n} \).
\end{lem}

\begin{proof}
To fix notation, let \( e_1,e_2 \) be the first two canonical basis vectors of \( \C^{2n+2} \), and embed \( \C^{2n} \) into \( \C^{2n+2} \) via the remaining coordinates. Assuming \( \GrSp{n}{2n} \) is defined in terms of a skew-symmetric form \( Q_{2n} \), define \( \GrSp{n+1}{2n+2} \) with respect to the form
\begin{align*}
 Q_{2n+2}:=\begin{pmatrix}
            0  & 1  & 0\\
            -1 & 0  & 0\\
            0  & 0  & Q_{2n}
           \end{pmatrix}
\end{align*}
Then we have embeddings \( i_1 \) and \( i_2 \) of \( \GrSp{n}{2n} \) into \( \GrSp{n+1}{2n+2} \) sending an \( n \)-plane \( \Lambda\subset\C^{2n} \) to \( e_1\oplus\Lambda \) or \( e_2\oplus\Lambda \) in \( \C^{2n+2} \), respectively.

We extend \( i_1 \) to an embedding of \( \vb{U}^{\perp}\oplus\OO \) by sending an \( n \)-plane \( \Lambda\in\GrSp{n}{2n} \) together with a vector \( v \) in \( \Lambda^{\perp}\subset\C^{2n} \) and a complex scalar \( z \) to the graph \( \Gamma_{\Lambda,v,z}\subset\C^{2n+2} \) of the linear map
\begin{equation*}
 \begin{pmatrix}
  z & Q_{2n}(-,v) \\
  v & 0
 \end{pmatrix}
 \colon{
  \left<e_1\right>\oplus\Lambda \rightarrow \left<e_{2}\right>\oplus\Lambda^{\perp}
 }
\end{equation*}
To avoid confusion, we emphasize that \( v \) is orthogonal to \( \Lambda \) with respect to a {\em hermitian} metric on \( \C^{2n} \). The value of \( Q_{2n}(-,v) \), on the other hand, may well be non-zero on \( \Lambda \). Consider the above embedding of \( \vb{U}^{\perp}\oplus\OO \) together with the embedding \( i_2 \):
\begin{equation*}
 \xymatrix@R=0pt{
  {\vb{U}^{\perp}\oplus\OO} \ar@{}[r]|{\longhookrightarrow}      & {\GrSp{n+1}{2n+2}}                        &  {\GrSp{n}{2n}} \ar@{}[l]|{\longhookleftarrow}_{i_2} \\
  {(\Lambda,v,z) }         \ar@{}[r]|{\mapsto}                     & {\Gamma_{\Lambda,v,z}}                                                     \\
                                                                   & {\left<e_2\right>\oplus\Lambda}       & \ar@{}[l]|{\mapsfrom} {\Lambda}
 }
\end{equation*}
To see that the two embeddings are complementary, take an arbitrary \( (n+1) \)-plane \( W \) in \( \GrSp{n+1}{2n+2} \).
If \( e_2\in W \) then we can consider a basis
\begin{equation*}
 e_2,\mm{a_1\\0\\v_1}, \dots, \mm{a_n\\0\\v_n}
\end{equation*}
of \( W \), and the fact that \( Q_{2n+2} \) vanishes on \( W \) implies that all \( a_i \) are zero. Thus \( W \) can be identified with \( i_2(\left<v_1,\dots,v_n\right>) \).

If, on the other hand, \( e_2 \) is not contained in \( W \) then we must have a vector of the form
\( {}^t(1, z', v') \)
in \( W \), for some  \( z'\in\C \) and \( v'\in\C^{2n} \). Extend this vector to a basis of \( W \) of the form
\begin{equation*}
 \mm{1\\z'\\v'},\mm{0\\b_1\\v_1}, \dots, \mm{0\\b_n\\v_n}
\end{equation*}
and let \( \Lambda:=\left<v_1,\ldots,v_n\right> \). The condition that \( Q_{2n+2} \) vanishes on \( W \) implies that \( Q \) vanishes on \( \Lambda \) and that \( b_i=Q_{2n}(v_i,v') \) for each \( i \). In particular, \( \Lambda \) is \( n \)-dimensional.  Moreover, we can replace the first vector of our basis by a vector
\( {}^t(1,z,v) \)
with \( v\in\Lambda^{\perp} \), by subtracting appropriate multiples of the remaining basis vectors. Since \( Q \) vanishes on \( \Lambda \) we have \( Q_{2n}(v_i,v')=Q_{2n}(v_i,v) \) and our new basis has the form
\begin{equation*}
 \mm{1\\z\\v},\mm{0\\Q(v_1, v)\\v_1}, \dots, \mm{0\\Q(v_n, v)\\v_n}
\end{equation*}
This shows that \( W=\Gamma_{\Lambda,v,z} \).
\end{proof}

\begin{cor}\label{cor:GrSp-CoSeq}
 We have a cofibration sequence
 \begin{equation*}
  \GrSp{n}{2n}_+\overset{i}{\hookrightarrow}\GrSp{n+1}{2n+2}_+\overset{p}{\twoheadrightarrow} \Thom_{\GrSp{n}{2n}}(\vb{U}^{\perp}\oplus\OO)
 \end{equation*}
\end{cor}
The associated long exact cohomology sequence splits into a short exact sequence of \( H^*(\GrSp{n+1}{2n+2}) \)-modules since all cohomology here is concentrated in even degrees:
\begin{equation}\label{seq:GrSp_ses}
 0\rightarrow \rH^*(\Thom_{\GrSp{n}{2n}}(\vb{U}^{\perp}\oplus\OO)) \overset{p^*}{\rightarrow} H^*(\GrSp{n+1}{2n+2})\overset{i^*}{\rightarrow} H^*(\GrSp{n}{2n})\rightarrow 0
\end{equation}
\begin{lem} Let \( c_i \) denote the \( i^{\text{th}} \) Chern classes of \( \vb{U} \) over \( \GrSp{n}{2n} \). We have
  \begin{align*}
   & H^*(\GrSp{n}{2n},\Sq^2)=\begin{cases}
                               \Lambda(a_1,a_5,a_9,\ldots,a_{4m-3}) & \text{ if \( n=2m \)}  \\
                               \Lambda(a_1,a_5,a_9,\ldots,a_{4m-3},a_{4m+1}) & \text{ if \( n=2m+1 \)}
                             \end{cases}\\
   & H^*(\GrSp{n}{2n},{\Sq^2}+c_1)=\begin{cases}
                                    \Lambda(a_1,a_5,\ldots,a_{4m-3})\cdot c_{2m} & \text{ if \( n=2m \)}\\
                                    0                                            & \text{ if \( n \) is odd}
                                   \end{cases}
  \end{align*}
  for certain generators \( a_i \) of degree \( 2i \).
\end{lem}
\begin{proof}
Consider the short exact sequence~\eqref{seq:GrSp_ses}.
The mod-2 cohomology of \( \GrSp{n}{2n} \) is an exterior algebra on the Chern classes \( c_i \) of \( \vb{U} \),
\begin{equation*}
 H^*(\GrSp{n}{2n};\Z/2)=\Lambda(c_1,c_2,\ldots,c_n)
\end{equation*}
and \( i^* \) is given by sending \( c_{n+1} \) to zero.
Thus, \( p^* \) is the unique morphism of \( H^*(\GrSp{n+1}{2n+2};\Z/2) \)-modules that sends the Thom class \( \theta \) of \( \vb{U}^{\perp}\oplus\OO \) to \( c_{n+1} \).

This short exact sequence induces a long exact sequence of cohomology groups with respect to the Steenrod square \( \Sq^2 \). The algebra \( H^*(\GrSp{n}{2n},\Sq^2) \) was computed in \cite{KonoHara:HSS}*{2--2}, with the result displayed above, so we already know two thirds of this sequence. Explicitly, we have \( a_{4i+1}=c_{2i}c_{2i+1} \),\footnote{In \cite{KonoHara:HSS} the generators are written as \( c_{2i}c'_{2i+1} \) with \( c'_{2i+1}=c_{2i+1}+c_1c_{2i} \).} so \( i^* \) is the obvious surjection sending \( a_i \) to \( a_i \) (or to zero). Thus, the long exact sequence once again splits.

If \( n=2m \) we obtain a short exact sequence
\begin{multline*}
 0\rightarrow H^*(\GrSp{2m}{4m},{\Sq^2}+c_1)\cdot\theta \overset{p^*}{\rightarrow} \Lambda(a_1,\ldots,a_{4m-3},a_{4m+1}) \\ \overset{i^*}{\rightarrow}\Lambda(a_1,\ldots,a_{4m-3})\rightarrow 0
\end{multline*}
We see that \( H^*(\GrSp{2m}{4m},{\Sq^2}+c_1)\cdot\theta \) is isomorphic to \mbox{\( \Lambda(a_1,\ldots,a_{4m-3})\cdot a_{4m+1} \)} as a  \( \Lambda(a_1,\ldots,a_{4m+1}) \)-module.
It is thus generated by a single element, which is the unique element of degree \( 8m+2 \).
Since \( p^*(c_{2m}\theta)=a_{4m+1} \), the class of \( c_{2m}\theta \) is the element we are looking for, and the result displayed above follows.

If, on the other hand, \( n \) is odd, then \( i^* \) is an isomorphism and \( H^*(\GrSp{n}{2n},{{\Sq^2}+c_1}) \) must be trivial.
\end{proof}

We see from the proof that \( p^* \) induces an injection of \mbox{\( H^*(\GrSp{n}{2n},{\Sq^2}+c_1)\cdot\theta \)} into \mbox{\( H^*(\GrSp{n}{2n},\Sq^2) \)}. Since we already know from \cite{KonoHara:HSS}*{Theorem~2.1} that the Atiyah-Hirzebruch spectral sequence for \( \KO^*(\GrSp{n}{2n}) \) collapses, we can apply Corollary~\ref{cor:Xcollaps-Tcollapse} to deduce that the spectral sequence for \mbox{\( \rKO^*(\Thom_{\GrSp{n}{2n}}(\vb{U}^{\perp}\oplus\OO)) \)} collapses at the \( E_3 \)-page as well.
This completes the proof of Theorem~\ref{thm:GrSp-KO}.

\subsection{Quadrics}\label{sec:eg:Q}
We next consider smooth complex quadrics \( Q^n \) in \( \P^{n+1} \). As far as we are aware, the first complete results on (shifted) Witt groups of split quadrics were due to Walter: they are mentioned together with the results for projective bundles in \cite{Walter:TGW} as the main applications of that paper. Unfortunately, they seem to have remained unpublished. Partial results are also included in Yagita's preprint \cite{Yagita:Q}, see~Corollary~8.3. More recently, Nenashev obtained almost complete results by considering the localization sequences arising from the inclusion of a linear subspace of maximal dimension \cite{Nenashev:Q}. Calm{\`e}s informs me that the geometric description of the boundary map given in \cite{BalmerCalmes:GeometricBoundary} can be used to show that these localization sequences split in general, yielding a complete computation. The calculation described here is completely independent of these results.

For \( n\geq3 \) the Picard group of \( Q^n \) is free abelian on a single generator given by the restriction of the universal line bundle \( \OO(1) \) over \( \P^{n+1} \). We will use the same notation \( \OO(1) \) for this restriction.

\begin{thm}\label{thm:Q-KO}
The KO-theory of a smooth complex quadric \( Q^n \) of dimension \( n\geq3 \) is as described in Table~\ref{table:results:Q}.

\begin{table}[btp]
\begin{center}
\begin{tabular}{>{\( }l<{ \)}|MM|MMMM|MMMM}
\toprule
\KO^*(Q^n;\lb{L}) &         &         & \multicolumn{4}{M|}{\vb{L}\equiv\OO}  & \multicolumn{4}{M}{\vb{L}\equiv\OO(1)} \\
                 & t_0     & t_1     & s_0 & s_1 & s_2 & s_3                & s_0 & s_1 & s_2 & s_3 \\
\midrule
n\equiv 0 \mod 8 & (n/2)+2 & n/2     & 2 & 0 & 0 & 0                        & 2 & 0 & 0 & 0 \\
n\equiv 1        & (n+1)/2 & (n+1)/2 & 1 & 1 & 0 & 0                        & 1 & 1 & 0 & 0 \\
n\equiv 2        & (n/2)+1 & (n/2)+1 & 1 & 2 & 1 & 0                        & 0 & 0 & 0 & 0 \\
n\equiv 3        & (n+1)/2 & (n+1)/2 & 1 & 1 & 0 & 0                        & 0 & 0 & 1 & 1 \\
\hline
n\equiv 4        & \minrowheight{\( () \)}(n/2)+2 & n/2   & 2 & 0 & 0 & 0                        & 0 & 0 & 2 & 0 \\
n\equiv 5        & (n+1)/2 & (n+1)/2 & 1 & 0 & 0 & 1                        & 0 & 1 & 1 & 0 \\
n\equiv 6        & (n/2)+1 & (n/2)+1 & 1 & 0 & 1 & 2                        & 0 & 0 & 0 & 0 \\
n\equiv 7        & (n+1)/2 & (n+1)/2 & 1 & 0 & 0 & 1                        & 1 & 0 & 0 & 1 \\
\bottomrule
\end{tabular}
\caption{KO-groups of projective quadrics (\( n\geq 3 \))}
\label{table:results:Q}
\end{center}
\end{table}
\end{thm}

\paragraph{Untwisted KO-groups.} Before turning to \( \KO^*(Q^n;\OO(1)) \) we review the initial steps in the computation of the untwisted KO-groups. The integral cohomology of \( Q^n \) is well-known:

If \( n \) is even, write \( n=2m \).
We have a class \( x \) in \( H^2(Q^n) \) given by a hyperplane section, and two classes \( a \) and \( b \) in \( H^n(Q^n) \) represented by linear subspaces of \( Q \) of maximal dimension. These three classes generate the cohomology multiplicatively, modulo the relations
\begin{align*}
 &x^m=a+b
 &&x^{m+1}=2ax \\
 &ab=\begin{cases}
     0    & \text{ if \( n\equiv 0 \)}\\
     ax^m & \text{ if \( n\equiv 2 \)}
    \end{cases}
 &&a^2=b^2=\begin{cases}
           ax^m & \text{ if \( n\equiv 0 \mod 4 \)}\\
           0    & \text{ if \( n\equiv 2 \mod 4 \)}
          \end{cases}
\end{align*}
Additive generators can thus be given as follows:
\begin{center}
\begin{tabular}[t]{M|MMMMMMMMMM}
d        & 0 & 2 & 4   & \dots & n-2     & n   & n+2 & n+4  & \dots & 2n    \\
\midrule
H^d(Q^n) & 1 & x & x^2 & \dots & x^{m-1} & a,b & ax  & ax^2 & \dots & ax^m
\end{tabular}
\end{center}
If \( n \) is odd, write \( n=2m+1 \). Then similarly multiplicative generators are given by the class of a hyperplane section \( x \) in \( H^2(Q^n) \) and the class of a linear subspace \( a \) in \( H^{n+1}(Q^n) \) modulo the relations \( x^{m+1}=2a \) and \( a^2 = 0 \).
\begin{center}
\begin{tabular}[t]{M|MMMMMMMMMM}
d        & 0 & 2 & 4   & \dots & n-1   & n+1 & n+3 & n+5  & \dots & 2n    \\
\midrule
H^d(Q^n) & 1 & x & x^2 & \dots & x^{m} & a   & ax  & ax^2 & \dots & ax^m
\end{tabular}
\end{center}
The action of the Steenrod square on \( H^{\ast}(Q^n;\Z/2) \) is also well-known; see for example \cite{Ishitoya:Squaring}*{Theorem~1.4 and Corollary~1.5} or \cite{EKM:Q}*{\S~78}:
\begin{align*}
 \Sq^2(x) &= x^2\\
 \Sq^2(a) &=\begin{cases}
                       ax & \text{ if \( n\equiv 0 \) or \( 3 \mod 4 \)} \\
                       0  & \text{ if \( n\equiv 1 \) or \( 2 \)}
                      \end{cases}\\
 \Sq^2(b) &= \Sq^2(a) \quad (\text{for even \( n \)})
\end{align*}

As before, we write \( H^{\ast}(Q^n,\Sq^2) \) for the cohomology of \( H^{\ast}(Q^n;\Z/2) \) with respect to the differential \( \Sq^2 \).
\begin{lem}\label{cor:Q-KO_untwisted}
Write \( n=2m \) or \( n=2m+1 \) as above. The following table gives a complete list of the additive generators of \( H^{\ast}(Q^n,\Sq^2) \).
\begin{center}
  \begin{tabular}[h]{r|MMMMMMMl}
  \( d \)              & 0 & \dots & n-1 & n   & n+1 & \dots & 2n   & \\
  \midrule
  \( H^d(Q^n,\Sq^2) \) & 1 &       &     &     &     &       & ax^m & if \( n\equiv 0 \mod 4 \)\\
                   & 1 &       &     &     & a   &       &      & if \( n\equiv 1 \)\\
                   & 1 &       &     & a,b &     &       & ab   & if \( n\equiv 2 \)\\
                   & 1 &       & x^m &     &     &       &      & if \( n\equiv 3 \)
  \end{tabular}
\end{center}
\end{lem}

The results of Kono and Hara on \( \KO^*(Q) \) follow from here provided there are no non-trivial higher differentials in the Atiyah-Hirzebruch spectral sequence. This is fairly clear in all cases except for the case \( n\equiv 2 \mod 4 \). In that case, the class \( a+b=x^m \) can be pulled back from \( Q^{n+1} \), and therefore all higher differentials must vanish on \( a+b \). But one has to work harder to see that all higher differentials vanish on \( a \) (or \( b \)).
Kono and Hara proceed by relating the KO-theory of \( Q^n \) to that of the spinor variety \( \GrSOc{\frac{n}{2}+1}{n+2} \) discussed in Section~\ref{sec:eg:spinor}.

\paragraph{Twisted KO-groups.}
We now compute \( \KO^{\ast}(Q^n;\OO(1)) \).

Let \( \theta\in H^2(\Thom_{Q^n}\!\OO(1)) \) be the Thom class of \( \OO(1) \), so that multiplication by \( \theta \) maps the cohomology of \( Q^n \) isomorphically to the reduced cohomology of \( \Thom_{Q^n}\!{\OO(1)} \).
The Steenrod square on \( \rH^*(\Thom_{Q^n}\!\OO(1);\Z/2) \) is determined by Lemma~\ref{lem:Sq-on-Thom}: for any \( y\in H^*(Q^n;\Z/2) \) we have \( \Sq^2(y\cdot\theta)=(\Sq^2{y}+xy)\cdot\theta \). We thus arrive at

\begin{lem}\label{lem:Q-KO_twisted}
The following table gives a complete list of the additive generators of \( \rH^{\ast}(\Thom_{Q^n}\!\OO(1),\Sq^2) \).
 \begin{center}
 \begin{tabular}[h]{r|MMMMMMl}
  \( d \)              &  \dots  & n+1       & n+2               & n+3     & \dots & 2n+2       & \\
  \midrule
  \( \rH^d(\dots) \)   &         &           & a\theta, b\theta  &         &       &            & if \( n\equiv 0 \mod 4 \)\\
                   &         & x^m\theta &                   &         &       & ax^m\theta & if \( n\equiv 1 \)\\
                   &         &           &                   &         &       &            & if \( n\equiv 2 \)\\
                   &         &           &                   & a\theta &       & ax^m\theta & if \( n\equiv 3 \)
 \end{tabular}
 \end{center}
\end{lem}

We claim that all higher differentials in the Atiyah-Hirzebruch spectral sequence for \( \rKO^{\ast}(\Thom_{Q^n}\!\OO(1)) \) vanish.
For even \( n \) this is clear.
But for \( n=8k+1 \) the differential \( d_{8k+2} \) might a priori take \( x^m\theta \) to \( ax^m\theta \),
and for \( n=8k+3 \) the differential \( d_{8k+2} \) might take \( a\theta \) to \( ax^m\theta \).

We therefore need some geometric considerations.
Namely, the Thom space \( \Thom_{Q^n}\!\OO(1) \) can be identified with the projective cone over \( Q^n \) embedded in \( \P^{n+2} \). This projective cone can be realized as the intersection of a smooth quadric \( Q^{n+2}\subset\P^{n+3} \) with its projective tangent space at the vertex of the cone \cite{Harris:AG}*{p.~283}. Thus, we can consider the following inclusions:
\begin{equation*}
 Q^n\overset{j}\hookrightarrow\Thom_{Q^n}\!\OO(1)\overset{i}\hookrightarrow Q^{n+2}
\end{equation*}
The composition is the inclusion of the intersection of \( Q^{n+2} \) with two transversal hyperplanes.

\begin{lem}
 All higher differentials (\( d_k \) with \( k>2 \)) in the Atiyah-Hirzebruch spectral sequence for \( \KO^{\ast}(\Thom_{Q^n}\!\OO(1)) \) vanish.
\end{lem}
\begin{proof}
We need only consider the cases when \( n \) is odd. Write \( n=2m+1 \).

When \( n\equiv 1 \) mod \( 4 \) we claim that \( i^* \) maps \( x^{m+1} \) in \( H^{n+1}(Q^{n+2},\Sq^2) \) to \( x^m\theta \) in \( H^{n+1}(\Thom_{Q^n}\!\OO(1),\Sq^2) \).
Indeed, \( j^*i^* \) maps the class of the hyperplane section \( x \) in \( H^2(Q^{n+2}) \) to the class of the hyperplane section \( x \) in \( H^2(Q^n) \).
So \( i^*x \) in \( H^2(\Thom_{Q^n}\!\OO(1)) \) must be non-zero, hence equal to \( \theta \) modulo \( 2 \).
It follows that \( i^*(x^{m+1})=\theta^{m+1} \). Since \( \theta^2=\Sq^2(\theta)=x\theta \), we have \( \theta^{m+1}=x^m\theta \), proving the claim. As we already know that all higher differentials vanish on \( H^{\ast}(Q^{n+2},\Sq^2) \), we may now deduce that they also vanish on \( H^{\ast}(\Thom_{Q^n}\!\OO(1),\Sq^2) \).

When \( n\equiv 3 \) mod \( 4 \) we claim that \( i^* \) maps \( a \) in \( H^{n+3}(Q^{n+2},\Sq^2) \) to
\( a\theta \) in \( H^{n+3}(\Thom_{Q^n}\!\OO(1),\Sq^2) \). Indeed, \( a \) represents a linear subspace of codimension \( m+2 \) in \( Q^{n+2} \) and is thus mapped to the class of a linear subspace of the same codimension in \( Q^n \): \( j^*i^*(a)=ax \) in \( H^{n+3}(Q^n) \). Thus, \( i^*(a) \) is non-zero in \( H^{n+3}(\Thom_{Q^n}\!\OO(1)) \), equal to \( a\theta \) modulo \( 2 \). Again, this implies that all higher differentials vanish on \( H^{\ast}(\Thom_{Q^n}\!\OO(1),\Sq^2) \) since they vanish on \( H^{\ast}(Q^{n+2},\Sq^2) \).
\end{proof}

The additive structure of \( \KO^*(Q^n;\OO(1)) \) thus follows directly from the result for \( H^d(Q^n,{\Sq^2}+x)=\rH^{d+2}(\Thom_{Q^n}\!\OO(1)) \) displayed in Lemma~\ref{lem:Q-KO_twisted} via Corollary~\ref{cor:2-torsion_of_Thom-KO}.

\subsection{Spinor varieties}\label{sec:eg:spinor}
Let \( \Gr_{\SO}(n,N) \) be the Grassmannian of \( n \)-planes in \( \C^N \) isotropic with respect to a fixed non-degenerate symmetric bilinear form, or, equivalently, the Fano variety of projective \( (n-1) \)-planes contained in the quadric \( Q^{N-2} \). For each \mbox{\( N>2n \)}, this is an irreducible homogeneous variety.
In particular, for \( N=2n+1 \) we obtain the spinor variety \( \GrSOc{n+1}{2n+2}=\Gr_{\SO}(n,2n+1) \).
The variety \( \Gr_{\SO}(n,2n) \) falls apart into two connected components, both of which are isomorphic to \( \GrSOc{n}{2n} \). This is reflected by the fact that we can equivalently identify \( \GrSOc{n}{2n} \) with \( \SO(2n-1)/U(n-1) \) or \( \SO(2n)/U(n) \).

As for all Grassmannians, the Picard group of \( \GrSOc{n}{2n} \) is isomorphic to \( \Z \); we fix a line bundle \( \mathcal{S} \) which generates it. The KO-theory twisted by \( \mathcal{S} \) vanishes:
\pagebreak
\begin{thm}\label{thm:GrSO-KO}
For all \( n\geq 2 \) the additive structure of \( \KO^*(\GrSOc{n}{2n};\lb{L}) \) is as follows:\vspace{-0.5\baselineskip}
 \begin{center}
 \begin{tabular}[t]{>{\( }l<{ \)}|MM|M|M}
 \toprule
                              & t_0     & t_1     & s_i(\OO)                             & s_i(\mathcal{S})          \\
   \midrule
   \text{\( n\equiv 2 \mod 4 \)}
     & 2^{n-2} & 2^{n-2} & \rho(\tfrac{n}{2}, {1-i}) & 0                \\
   \minrowheight{\( \tfrac{n}{2} \)}
   \text{otherwise}
     & 2^{n-2} & 2^{n-2} & \rho(\lfloor\tfrac{n}{2}\rfloor, {-i})  & 0  \\
   \bottomrule
 \end{tabular}
 \end{center}

 \vspace{0.5\baselineskip}
 \noindent
 Here, the values \( \rho(n,i) \) are defined as in Theorem~\ref{thm:GrSp-KO}.
\end{thm}
\begin{proof}
The cohomology of \( \GrSOc{n}{2n} \) with \( \Z/2 \)-coefficients has simple generators \( e_2 \), \( e_4 \), \dots, \( e_{2n-2} \), \ie it is additively generated by products of distinct elements of this list.
Its multiplicative structure is determined by the rule \( e_{2i}^2=e_{4i} \), and the second Steenrod square is given by \( \Sq^2(e_{2i})=ie_{2i+2} \) \cite{Ishitoya:Squaring}*{Proposition~1.1}.
In both formulae it is of course understood that \( e_{2j}=0 \) for \( j\geq n \).
What we need to show is that for all \( n\geq 2 \) we have
\begin{equation*}
 H^*(\GrSOc{n}{2n},{\Sq^2}+e_2) = 0
\end{equation*}
Let us abbreviate \( H^*(\GrSOc{n}{2n},{\Sq^2}+e_2) \) to \( (H_n,d') \).
We claim that we have the following short exact sequence of differential \( \Z/2 \)-modules:
\begin{equation}\label{seq:Spinor-SqCohomology}
 0 \rightarrow (H_n,d') \overset{\cdot e_{2n}}{\rightarrow} (H_{n+1},d') \rightarrow (H_n,d') \rightarrow 0
\end{equation}
This can be checked by a direct calculation.
Alternatively, it can be deduced from the geometric considerations below.
Namely, it follows from the cofibration sequence of Corollary~\ref{cor:Spinor-CoSeq} that we have such an exact sequence of \( \Z/2 \)-modules with maps respecting the differentials given by \( \Sq^2 \) on all three modules.
Since they also commute with multiplication by \( e_2 \), they likewise respect the differential \( d'={\Sq^2}+e_2 \).

The long exact cohomology sequence associated with \eqref{seq:Spinor-SqCohomology} allows us to argue by induction: if \( H^*(H_n,d')=0 \) then  also \( H^*(H_{n+1},d')=0 \). Since we can see by hand that \( H^*(H_2,d')=0 \), this completes the proof.
\end{proof}

We close with a geometric interpretation of the exact sequence~\eqref{seq:Spinor-SqCohomology}, via an analogue of Lemmas \ref{lem:Gr-Geometry} and \ref{lem:GrSp-Geometry}.
Let us write \( \vb{U} \) for the universal bundle over \( \GrSOc{n}{2n} \), \ie for the restriction of the universal bundle over \( \Gr(n-1,2n-1) \) to \( \GrSOc{n}{2n} \), and \( \vb{U}^{\perp} \) for the restriction of the orthogonal complement bundle, so that \( \vb{U}\oplus\vb{U}^{\perp} \) is the trivial (\( 2n-1 \))-bundle over \( \GrSOc{n}{2n} \). As in Section~\ref{sec:eg:symplecticGr}, we emphasize that under these conventions the fibres of \( \vb{U} \) and \( \vb{U}^{\perp} \) are perpendicular with respect to a hermitian metric on \( \C^{2n-1} \) --- they are not orthogonal with respect to the chosen symmetric form.
\begin{lem}\label{lem:Spinor-Geometry}
The spinor variety \( \GrSOc{n}{2n} \) embeds into the spinor variety \( \GrSOc{n+1}{2n+2} \) with normal bundle \mbox{\( \vb{U}^{\perp} \)} such that the embedding extends to an embedding of this bundle. The closed complement of \mbox{\( \vb{U}^{\perp} \)} in \( \GrSOc{n+1}{2n+2} \) is again isomorphic to \( \GrSOc{n}{2n} \).
\end{lem}

\begin{cor}\label{cor:Spinor-CoSeq}
 We have a cofibration sequence
 \begin{equation*}
  \GrSOc{n}{2n}_+\overset{i}{\hookrightarrow}\GrSOc{n+1}{2n+2}_+\overset{p}{\twoheadrightarrow} \Thom_{\GrSOc{n}{2n}}{\vb{U}^{\perp}}
 \end{equation*}
\end{cor}

Note however that, unlike in the symplectic case, the first Chern classes of \( \vb{U} \) and \( \vb{U}^{\perp} \) pull back to {\em twice} a generator of the Picard group of \( \GrSOc{n}{2n} \). For example, the embedding of \( \GrSOc{2}{4} \) into \( \Gr(1,3) \) can be identified with the embedding of the one-dimensional smooth quadric into the projective plane, of degree \( 2 \), and the higher dimensional cases can be reduced to this example.
Thus, \( c_1(\vb{U}) \) and \( c_1(\vb{U}^{\perp}) \) are trivial in \( \Pic(\GrSOc{n}{2n})/2 \).

\begin{proof}[proof of Lemma~\ref{lem:Spinor-Geometry}]
The proof is similar to the proof of Lemma~\ref{lem:GrSp-Geometry}. Let \( e_1, e_2 \) be the first two canonical basis vectors of \( \C^{2n+1} \), and let \( \C^{2n-1} \) be embedded into \( \C^{2n+1} \) via the remaining coordinates. Let \( \GrSOc{n}{2n} \) be defined in terms of a symmetric form \( Q \) on \( \C^{2n-1} \), and define \( \GrSOc{n+1}{2n+2} \) in terms of
\begin{align*}
 Q_{2n+1}:=\begin{pmatrix}
            0  & 1  & 0\\
            1 & 0  & 0\\
            0  & 0  & Q
           \end{pmatrix}
\end{align*}
Let \( i_1 \) and \( i_2 \) be the embeddings of \( \GrSOc{n}{2n} \) into \( \GrSOc{n+1}{2n+2} \) sending an \( (n-1) \)-plane \( \Lambda\subset\C^{2n-1} \) to \( e_1\oplus\Lambda \) or \( e_2\oplus\Lambda \) in \( \C^{2n+1} \), respectively.
Given an \( (n-1) \)-plane \( \Lambda\in\GrSOc{n}{2n} \) together with a vector \( v \) in \( \Lambda^{\perp}\subset\C^{2n-1} \), consider the linear map
\begin{equation*}
 \begin{pmatrix}
  -\tfrac{1}{2}Q(v,v) & -Q(-,v) \\
  v                   & 0
 \end{pmatrix}
 \colon{
  \left<e_1\right>\oplus\Lambda \rightarrow \left<e_{2}\right>\oplus\Lambda^{\perp}
 }
\end{equation*}
Sending \( (\Lambda,v) \) to the graph of this function defines an open embedding of \( \vb{U}^{\perp} \) whose closed complement is the image of \( i_2 \).
\end{proof}

\subsection{Exceptional hermitian symmetric spaces}\label{sec:eg:exceptional}
Lastly, we turn to the exceptional hermitian symmetric spaces \( \EIII \) and \( \EVII \). We write \( \OO(1) \) for a generator of the Picard group in both cases.

\begin{thm}\label{thm:Exceptional-KO}
The KO-groups of the exceptional hermitian symmetric spaces \( \EIII \) and \( \EVII \) are as follows:\vspace{-0.5\baselineskip}
\begin{center}
  \begin{tabular}[t]{>{\( }l<{ \)}|MM|MMMM|MMMM}
  \toprule
   &  &  & \multicolumn{4}{M|}{\vb{L}\equiv\OO}  & \multicolumn{4}{M}{\vb{L}\equiv\OO(1)} \\
   & t_0  & t_1 & s_0 & s_1 & s_2 & s_3  & s_0 & s_1 & s_2 & s_3 \\
  \midrule
  \KO^*(\EIII;\lb{L}) & 15 & 12 & 3 & 0 & 0 & 0 & 3 & 0 & 0 & 0 \\
  \KO^*(\EVII;\lb{L}) & 28 & 28 & 1 & 3 & 3 & 1 & 0 & 0 & 0 & 0 \\
  \bottomrule
  \end{tabular}
\end{center}
\end{thm}

\vspace{0.5\baselineskip}
\begin{proof}
The untwisted KO-groups have been computed in \cite{KonoHara:HSS}, the main difficulty as always being to prove that the Atiyah-Hirzebruch spectral sequence collapses. For the twisted groups, however, there are no problems.
We quote from \S~3 of said paper that the cohomologies of the spaces in question can be written as
\begin{align*}
H^*(\EIII;\Z/2)&=\Z/2\big[t,u\big]\Big/(u^2t,\;u^3+t^{12})\\
H^*(\EVII;\Z/2)&=\Z/2\big[t,v,w\big]\Big/(t^{14},\;v^2,\;w^2)
\intertext{with \( t \) of degree \( 2 \) in both cases, and \( u \), \( v \) and \( w \) of degrees \( 8,10 \) and \( 18 \), respectively. The Steenrod squares are determined by \( \Sq^2u=ut \) and \( \Sq^2v=\Sq^2w=0 \). Thus, we find}
H^*(\EIII,{\Sq^2}+t)&=\Z/2\cdot u\oplus\Z/2\cdot u^2\oplus\Z/2\cdot u^3\\
H^*(\EVII,{\Sq^2}+t)&=0
\end{align*}
By Lemma~\ref{lem:AHSS_first-differential} the Atiyah-Hirzebruch spectral sequence for \( \EIII \) must collapse. This gives the result displayed above.
\end{proof}

\subsection*{Acknowledgements}
The present paper is based on research carried out in the context of the author's PhD studies, which have been generously funded by the Engineering and Physical Sciences Research Council (EPSRC). I wish to thank my supervisor Burt Totaro for his commitment, time and patience, for directing me towards the problems and ideas discussed here, and for his continuing support. Special thanks are moreover due to Baptiste Calm{\`e}s and Marco Schlichting for highly helpful conversations and explanations of their respective work.
Finally, I am grateful to Jens Hornbostel for constructive comments on an earlier version of this paper and to the anonymous referee for careful proofreading.

\end{document}